\documentclass[a4paper,12pt]{article}

\usepackage{lmodern}
\usepackage[T1]{fontenc}
\usepackage{hyperref}
\usepackage{amsfonts,amsmath,amssymb,amsopn,amsthm}
\usepackage[all]{xy}
\usepackage{vmargin}
\usepackage{makeidx}
\usepackage{graphicx}
\usepackage{tikz-cd}
\usepackage{comment}
\usepackage{mathabx}
\usepackage{mathbbol}
\mathchardef\mhyphen="2D

\usepackage{pifont}
\title{$\mathbb M$-modules}
\author{Bernard Le Stum}
\date{Version of \today}
\usepackage[ style=alphabetic,citestyle=alphabetic,sorting=nyt,sortcites=true,autopunct=true,babel=hyphen,hyperref=true,abbreviate=false,backref=true]{biblatex}
\AtEveryBibitem{ \clearfield{url} \clearfield{urldate} \clearfield{review} \clearfield{series} \clearfield{doi} \clearfield{isbn} \clearfield{issn} \clearfield{note} }
\defbibheading{bibempty}{}
\DeclareNolabel{\nolabel{\regexp{[\p{Z}]+}}}
\newtheorem{thm}{Theorem}[section]
\newtheorem{prop}[thm] {Proposition}
\newtheorem{cor}[thm] {Corollary}
\newtheorem{lem}[thm] {Lemma}
\theoremstyle{definition}
\newtheorem{dfn}[thm] {Definition}

\newenvironment{xmp}[1][Example]{\begin{trivlist} \item[\hskip \labelsep {\bfseries #1}]}{\end{trivlist}}
\newenvironment{xmps}[1][Examples]{\begin{trivlist} \item[\hskip \labelsep {\bfseries #1}]}{\end{trivlist}}

\newenvironment{rmk}[1][Remark]{\begin{trivlist} \item[\hskip \labelsep {\bfseries #1}]}{\end{trivlist}}

\usepackage[pagewise]{lineno}

\newcommand{\Addresses}{{
 \bigskip
 \footnotesize

Bernard Le Stum, \textsc{IRMAR, Université de Rennes,
Campus de Beaulieu, 35042 Rennes cedex, France}\par\nopagebreak
\texttt{bernard.le-stum@univ-rennes1.fr}

}}

\begin{document}

\maketitle

\begin{abstract}
We consider the ring $\mathbb M$ of column-finite matrices with integer coefficients.
We prove that $\mathbb M$-modules form an additive closed symmetric monoidal abelian category that essentially contains complete metrizable linearly topologized abelian groups as a full subcategory.
We also introduce and discuss the notion of an $\mathbb M$-ring.
In the end, we show that the category of $\mathbb M$-modules (resp.\ $\mathbb M$-rings) is actually equivalent to the category of light solid abelian groups (resp.\ light solid rings) of Clausen and Scholze, which should come as no surprise to specialists.
However, we believe that our more straightforward approach may offer a fresh perspective on their theory. 
\end{abstract}

\tableofcontents

\section*{Introduction}

When I first heard about solid abelian groups, it quickly dawned on me that there existed a simpler way to describe them, namely as additive functors on free abelian groups\footnote{This perspective was eventually developed by Kiran Kedeya in \cite{Kedlaya25}.} -- exactly as a usual abelian group may be seen as an additive functor on free abelian groups of \emph{finite} rank.
When the light version came into the picture, I realized that the situation was even better: the category of light solid abelian groups is equivalent to the category of left modules over the ring $\mathbb M$ of column-finite matrices with integer coefficients.
The purpose of this note is to give a direct description of $\mathbb M$-modules allowing the reader to enter the world of solid mathematics without learning condensed mathematics first.
I want to emphasize, however, that there is no pretension to originality here. All our results can be easily derived from the original theory developed by Dustin Clausen and Peter Scholze in \cite{ClausenScholze23}.

Before going any further, let us mention that there exists a strong analogy between the matrix ring $\mathbb M$ and the Weyl algebra $\mathcal D$.
In the same way as we can write
\[
\mathcal D = \left\{\sum a_{kn}x^k\partial^n, \quad a_{kn} \in \mathbb C\right\}
\]
with 
\[
 x^kx^n=x^{n+k} , \quad \partial^n\partial^m =\partial^{n+m} \quad \mathrm{and} \quad \partial x^{n+1} = (n+1)x^n,
\]
we can interpret the ring of matrices as the \emph{$t$-adically complete} ring
\[
\mathbb M =\left\{\sum a_{kn}e_kt^n, \quad a_{kn} \in \mathbb Z\right\}
\]
with
\[
e_ke_n = \left\{\begin{array}{ll} e_k\ \mathrm{if}\ n=0 \\ 0\ \mathrm{otherwise}.\end{array}\right., \quad t^nt^m=t^{n+m} \quad \mathrm{and} \quad te_{n+1} = e_{n}.
\]
More precisely, we have on the one hand,
\[
\mathcal O = \left\{\sum_{\mathrm{finite}} c_n x^n, c_n \in \mathbb C\right\} \subset \mathcal D = \left\{\sum_{\mathrm{finite}} f_n \partial^n, f_n \in \mathcal O \right\} \subset \widehat{\mathcal D} = \left\{\sum_{\mathrm{infinite}} f_n \partial^n\right\},
\]
and on the other,
\[
\mathbb Z_\bullet := \left\{\sum_{\mathrm{finite}} c_n e_n, c_n \in \mathbb Z\right\} \subset \mathbb M^\vee := \left\{\sum_{\mathrm{finite}} a_n t^n, a_n \in \mathbb Z_\bullet \right\} \subset \mathbb M = \left\{\sum_{\mathrm{infinite}} a_n t^n \right\}
\]
(but $\mathbb Z_\bullet$ is now a left ideal in $\mathbb M$ and $\mathbb M^\vee$ is a two sided ideal).
Another similarity is the identifications $\mathrm{End}_{\mathbb C}(\mathcal O) \simeq \widehat {\mathcal D}$ as well as $\mathrm{End}_{\mathbb Z}(\mathbb Z_\bullet) \simeq \mathbb M$.
Also, there exists resolutions
\[
0 \to \mathcal D \overset {\cdot \partial} \to \mathcal D \to \mathcal O \to 0 \quad \mathrm{and} \quad 0 \to \mathbb M \overset {\cdot t} \to \mathbb M \to \mathbb Z_\bullet \to 0
\]
(where the second one splits).
Both categories of (left) modules over $\mathcal D$ and $\mathbb M$ are additive closed symmetric monoidal with unit $\mathcal O$ and $\mathbb Z_\bullet$ respectively.
Finally, the category of representations of the (punctured) fundamental group embeds into the category of $\mathcal D$-modules exactly as complete metrizable linearly topologized abelian groups will embed into the category of $\mathbb M$-modules.
And this is in both cases the main purpose of the theory: working inside the nice category of modules over a (non commutative) ring in order to better understand a more subtle notion.

Let us now set aside our analogy and attempt to elucidate precisely what $\mathbb M$-modules serve for.
Many mathematical problems can be formulated in terms of sheaves of abelian groups.
Although there usually exists some kind of underlying topological space, these sheaves have a purely algebraic flavor allowing powerful algebraic techniques.
We are however missing something because the topological background often provides these sheaves with a specific topology.
Unfortunately, if we take this topology into account, then we leave the realm of pure algebra and loose its remarquable tools.
For example, the modules that appear in non Archimedean geometry (say finitely presented modules over a Huber ring) are metrizable complete for a linear topology.
The category of (metrizable complete) linearly topologized abelian groups is symmetric monoidal (but \emph{not} closed) and preabelian (but \emph{not} abelian).
Using condensed mathematics, Dustin Clausen and Peter Scholze designed the category of (light) solid abelian groups as an enlargement of the former category that is closed symmetric monoidal abelian.
They define it as some reflective subcategory of the bigger category of (light) condensed abelian groups.
The purpose of the present note is to describe and study a flavor of light solid abelian groups that do not rely on condensed mathematics. 
We consider the ring $\mathbb M$ of column-finite matrices with integer coefficients and show that the abelian category of $\mathbb M$-modules can be made closed symmetric monoidal.
Then, if we are given a complete linear (we mean linearly topologized) abelian group, one can see a null sequence as a column vector and make $\mathbb M$ act on them.
This is well defined and functorial.
We obtain a monoidal embedding of the category of metrizable complete linear abelian groups into a closed symmetric monoidal abelian category.

Our approach may also be seen as a limit theory.
More precisely, if we denote by $\mathbb M_r$ the ring of matrices of finite rank $r$ with integer coefficients, then Morita equivalence tells us that the category of left $\mathbb M_r$-modules is equivalent to the category of abelian groups (whatever $r$).
When we allow infinite rows in our matrices (but still finite columns), we then replace $\mathbb M_r$ with $\mathbb M$ and it happens (theorem \ref{infMor}) that the category of left $\mathbb M$-modules is equivalent to the category of light solid abelian groups.
This is actually completely formal according to the results presented in the appendix but we shall also give an explicit construction.

\section*{Conventions}

We shall use von Neumann notation and write $n : = \{0, 1, \ldots, n-1\}$ and $\mathbb N := \{0, 1, 2, \ldots\} = \varinjlim n$.
When any of these sets is considered as an ordered set, then it is implicitly endowed with its usual order.
A \emph{sequence} is a family indexed by $\mathbb N$.

Unless otherwise specified, a usual set/abelian group will be endowed with the discrete topology and called a \emph{discrete} set/abelian group.
However, (infinite) limits of discrete sets/abelian groups will always be endowed with the limit topology.
This will for example be the case for $\overline {\mathbb N} := \{0, 1, 2, \ldots, \infty\} = \varprojlim \overline n$ with $\overline n =\{0, 1, \ldots, n-1, \infty\}$.

Given a set $I$ and an abelian group $V$, we shall denote by $V^I := \prod_I V$ the set of possibly infinite $I$-indexed sequences of elements of $V$ and by $V \cdot I := \oplus_I V$ the set of finite $I$-indexed sequences.
Instead of $V \cdot I$, it is also common to write $V^{(I)}$, or $I \otimes V$ or $IV$ or $V[I]$.
I had to rule out the last notation which is favored in condensed mathematics not to get confused with polynomial rings.

Unless otherwise specified, all rings are assumed to be associative and unitary (but not commutative).
We will usually simply say module instead of \emph{left} module.

We shall use a calligraphic letter and write $\mathcal H \mathrm{om}$ instead of $\underline {\mathrm{Hom}}$ to denote internal Hom not to get confused with the condensed set $\underline X$ associated to a topological space $X$.

We will follow Johnstone and call \emph{finitely presentable} - and not compact - an object $X$ such that $\mathrm{Hom}(X, -)$ preserves filtered colimits.
Following Johnstone again, a totally disconnected compact Hausdorff space will be called a \emph{Stone} space - and not a profinite set.
For example, we shall say ``metrizable Stone'' instead of the more fashionable ``light profinite'' -- which is the same thing.


\section{The matrix ring} \label{secmain}


\begin{dfn} \label{defM}
The \emph{(infinite) matrix ring} is the set
\[
\mathbb M := \{ [c_{kn}]_{k,n \in \mathbb N}, \forall n \in \mathbb N, \exists K \in \mathbb N, \forall k \geq K, c_{kn} = 0 \} \subset \mathbb Z^{\mathbb N \times \mathbb N}
\]
of column-finite matrices with integer coefficients endowed with usual multiplication
\[
[c_{kn}] \times [d_{kn}] = \left[\sum_i c_{ki}d_{in}\right].
\]
\end{dfn}

One easily checks that $\mathbb M$ is stable under addition, that multiplication is well defined and that $\mathbb M$ then becomes a non commutative ring (but see also proposition \ref{isoM} below).
We could 	as well consider the ring $\mathbb M'$ of row-finite matrices but transposition $u \mapsto u' := u^T$ then provides an isomorphism of rings $\mathbb M^{\mathrm{op}} \simeq \mathbb M'$.
 In particular, we can identify right $\mathbb M$-modules with left $\mathbb M'$-modules (and vice versa).

Anticipating on further developments, when $V$ is a \emph{discrete} abelian group, we shall denote by $V_\bullet$ the set of finite sequences of elements of $V$ (seen as \emph{column} vectors) so that $V_\bullet \simeq V \cdot \mathbb N := \oplus_{\mathbb N} V$.
For example, $\mathbb Z_\bullet \simeq \mathbb Z \cdot \mathbb N$ is the free abelian group on $(e_k)_{k \in \mathbb N}$ with $e_k = (0, \ldots, 0, 1, 0, \ldots)$ (written a a column) where $1$ sits at the $k$th rank\footnote{The $k+1$st place since we start counting at zero.}.
Our main players will then be
\[
\mathbb M = \left\{\left[\begin{array}{cccc} * & * & \cdots \\ \vdots & \vdots & \\ * & \vdots & \\ 0 & * & \\ \vdots & 0 & \\ \vdots & \vdots & \\ & \end{array} \right]\right\}
\quad \mathrm{and} \quad
\mathbb Z_\bullet := \left\{\left[\begin{array}{c} * \\ \vdots \\ * \\ 0 \\ \vdots \end{array} \right]\right\}.
\]
There exists a left module structure on the abelian group $\mathbb Z_\bullet$ which is given by
\[
[c_{kn}] \times (d_{k}) = \left(\sum_n c_{kn}d_{n}\right).
\]
In general, if $V$ is an abelian group, then
\[
V_\bullet \simeq \mathbb Z_\bullet \otimes_{\mathbb Z} V \quad (\mathrm{resp.} \
V^\mathbb N \simeq \mathrm{Hom}_{\mathbb Z}(\mathbb Z_\bullet, V))
\]
 becomes a left (resp.\ right) $\mathbb M$-module.

\begin{prop} \label{isoM}
The action of $\mathbb M$ on $\mathbb Z_\bullet$ (resp.\ $\mathbb Z^\mathbb N$) provides an isomorphism of rings
\[
\mathbb M \simeq \mathrm{End}_{\mathbb Z}(\mathbb Z_\bullet) \quad (\mathrm{resp.} \quad \mathbb M' \simeq \mathrm{End}_{\mathbb Z}(\mathbb Z^\mathbb N)).
\]
\end{prop}

\begin{proof}
The first assertion is easily checked and the second one relies on the fact that $\mathbb Z$ is \emph{slender} is the sense that $\mathrm{Hom}_{\mathbb Z}(\mathbb Z^\mathbb N, \mathbb Z) \simeq \mathbb Z \cdot \mathbb N$
(corollary 94.6 of \cite{Fuchs73}).
\end{proof}

Actually, it is almost a triviality that the first isomorphism holds as an isomorphism of abelian groups and this can be used to transfer the ring structure of $\mathrm{End}_{\mathbb Z}(\mathbb Z_\bullet)$ to $\mathbb M$ (and avoid tedious verifications after definition \ref{defM} above).

Throughout, we shall identify $\mathbb Z_\bullet$ (resp.\ $\mathbb Z^\mathbb N$) with the first column (resp.\ row) of $\mathbb M$.

There exists an $\mathbb M$-linear isomorphism of left $\mathbb M$-modules $(\mathbb Z_\bullet)^\mathbb N \simeq \mathbb M$
(list of columns) which shows that $\mathbb M$ acts on itself -- on the left -- column by column.
In particular, $\mathbb Z_\bullet$ is a finitely generated projective left ideal of $\mathbb M$.

\begin{dfn} \label{cantop}
The \emph{canonical topology} on $\mathbb M$ is transferred from the product topology under the isomorphism $(\mathbb Z_\bullet)^\mathbb N \simeq \mathbb M$ when $\mathbb Z_\bullet$ itself is endowed with the discrete topology.
\end{dfn}

In other words, a matrix is small when many first columns are made of zeros.

We shall denote by $t \in \mathbb M \simeq \mathrm{End}_{\mathbb Z}(\mathbb Z_\bullet)$ the matrix corresponding to the shift
\[
te_k =\left\{\begin{array}{ll} 0\ \mathrm{if}\ k=0 \\ e_{k-1}\ \mathrm{otherwise}, \end{array}\right.
\]
and by $t'$ its transpose given by $t'e_k = e_{k+1}$ so that
\[
t := \left[\begin{array}{ccccc} 0 & 1 & 0 & \cdots \\ 0 & 0 & 1 & \\ \vdots & & \ddots & \ddots \end{array} \right]
\quad \mathrm{and} \quad
t' := \left[\begin{array}{ccccc} 0 & 0 & \cdots \\ 1 & 0 & 0 \\ 0 & 1 & \ddots \\ \vdots & & \ddots \end{array} \right].
\]
If $u \in \mathbb M$, then $tu$ is obtained from $u$ by shifting lines up, $ut$ by shifting columns right, $t'u$ by shifting lines down and $ut'$ by shifting columns left.
As a consequence we have $tt' = 1$ and $t't$ erases the first column (resp.\ row) if it acts on the left (resp.\ right).

\begin{prop} \label{split1}
There exists split exact sequences of left (resp.\ right) $\mathbb M$-modules
\begin{align}
\begin{tikzcd}[ampersand replacement=\&]
0 \ar[r] \& \mathbb M \ar[r, " \cdot t^n "'] \& \mathbb M \ar[r] \ar[l, " \cdot t'^n "', bend right] \& \mathbb Z_\bullet^n \ar[r] \ar[l, bend right] \& 0 \quad \mathrm{and}
\end{tikzcd}
\\
\begin{tikzcd}[ampersand replacement=\&]
0 \ar[r] \& (\mathbb Z^\mathbb N)^n \ar[r] \& \mathbb M \ar[r, " t^n \cdot "'] \ar[l, bend right] \& \mathbb M \ar[r] \ar[l, " t'^n \cdot "', bend right] \& 0
\end{tikzcd}
\end{align}
(the dot indicates on which side the matrix acts and the extra maps are projection on, or inclusion of, the first rows/columns).
\end{prop}

\begin{proof}
Left as an exercise.
\end{proof}

Said differently,
\[
\mathbb Z_\bullet^n \simeq \mathbb M/\mathbb Mt^n \simeq \mathbb M^{\cdot t'^n = 0}\quad \mathrm{(resp.\ }\ 
\mathbb Z^\mathbb N \simeq \mathbb M/t'^n\mathbb M \simeq \mathbb M^{t^n\cdot = 0}).
\]
Going one step further provides similar split exact sequences of abelian groups
\begin{align}
\begin{tikzcd}[ampersand replacement=\&]
0 \ar[r] \& \mathbb Z^n \ar[r] \& \mathbb Z_\bullet \ar[r, " t^n \cdot "'] \ar[l, bend right] \& \mathbb Z_\bullet \ar[r] \ar[l, " t'^n \cdot "', bend right] \ar[l, bend right]\& 0 \quad \mathrm{and}
\end{tikzcd}
\\
\begin{tikzcd}[ampersand replacement=\&]
0 \ar[r] \& \mathbb Z^\mathbb N \ar[r, " \cdot t^n "'] \& \mathbb Z^\mathbb N \ar[r] \ar[l, " \cdot t'^n "', bend right] \& \mathbb Z^n \ar[r] \ar[l, bend right] \& 0.
\end{tikzcd}
\end{align}

From now on, we will follow our general conventions and simply say \emph{module} (resp.\ \emph{ideal}) for left module (resp.\ left ideal).

\begin{xmps}
\begin{enumerate}
\item Both $\mathbb M$ and $\mathbb Z_\bullet$ are $\mathbb M$-modules but $\mathbb Z^\mathbb N$ is not: this is a right $\mathbb M$-module or, equivalently, an $\mathbb M'$-module.
\item
The subset
\[
\mathbb M^\vee := \{ [a_{kn}]_{k,n \in \mathbb N}, \exists N, k+n> N \Rightarrow a_{kn} = 0 \} \subset \mathbb M
\]
(use $\max(k,n)$ if you like it better) made of all finite matrices is a two-sided projective ideal (but not a subring since it lacks a unit).
As $\mathbb M$-modules, we have $(\mathbb Z \cdot \mathbb N)_\bullet \simeq \mathbb Z_\bullet \cdot \mathbb N \simeq \mathbb M^\vee$ but we also have $\mathbb M^\vee \simeq \mathbb Z \cdot (\mathbb N \times \mathbb N)$ as abelian group.
\item
We shall also encounter the sets $\mathbb M^{\otimes n}$ of hypermatrices of size $n$ made of families $[a_{i_0i_1\ldots i_n}]$ that are finite (only) in the first variable so that $(\mathbb Z_\bullet)^{\mathbb N^n} \simeq \mathbb M^{\otimes n}$.
They come with a (left) $\mathbb M$-module structure, as well as $n$ different equivalent right $\mathbb M$-module structures.
\item Go to section \ref{Seclin} for more examples.
\end{enumerate}
\end{xmps}

Recall that we equipped in defintion \ref{cantop} the matrix ring $\mathbb M$ with a topology.

\begin{prop} \label{tMod}
The canonical topology of $\mathbb M$ is the $t$-adic topology and $\mathbb M$ is complete (Hausdorff) for this topology.
\end{prop}

\begin{proof}
It means that the ideals $\mathbb Mt^n$ for $n \in \mathbb N$ form a basis of neighborhoods of zero and that
\[
\mathbb M \simeq \varprojlim_{n \in \mathbb N} \mathbb M/\mathbb Mt^n.
\]
This follows from proposition \ref{split1}.
\end{proof}

Being an ideal in $\mathbb M$, the free abelian group $\mathbb Z_\bullet$ on $(e_k)_{k \in \mathbb N}$ inherits the structure of a \emph{non unitary} ring with commutation rule
\begin{equation} \label{rul1}
e_ke_n = \left\{\begin{array}{ll} e_k\ \mathrm{if}\ n=0 \\ 0\ \mathrm{otherwise} \end{array}\right.
\end{equation}
(so that incidentally $e_0$ is actually a right unit).
Then
\[
\mathbb M = \mathbb Z_\bullet[[t]] = \left\{\sum_{n=0}^\infty a_n t^n, a_n \in \mathbb Z_\bullet \right\}
\]
is the non-commutative power series ring with coefficients in $\mathbb Z_\bullet$ and commutation rule
\begin{equation} \label{rul2}
te_k =\left\{\begin{array}{ll} 0\ \mathrm{if}\ k=0 \\ e_{k-1}\ \mathrm{otherwise}. \end{array}\right.
\end{equation}
In particular, any element of $\mathbb M$ can be uniquely written $\sum a_{kn}e_kt^n$ as an infinite series in $n$ (but finite in $k$ for fixed $n$) with $a_{kn} \in \mathbb Z$.
Actually, $e_kt^n$ is nothing but the matrix with $1$ on $k+1$st row and $n+1$st column and $0$ everywhere else and
\[
(e_kt^n)(e_\ell t^m) = \left\{ \begin{array}{ll} e_kt^m & \mathrm{if}\ n = \ell \\ 0 & \mathrm{otherwise}. \end{array}\right.
\]
Note that $\mathbb M^\vee = \mathbb Z_\bullet[t]$ is a \emph{non unitary} non commutative polynomial ring whose $t$-adic completion $\mathbb M$ becomes unitary with $1 = \sum_{n=0}^\infty e_nt^n$.
We also have $\mathbb M^{\otimes n} = \mathbb Z_\bullet[[t_1, \ldots, t_n]]$ where $t_i$ denotes the shift in the $i$-th direction - but this is \emph{not} a ring in any way when $n >1$.

Since $\mathbb M$ contains $\mathbb Z[[t]]$ as a non-central subring, there exists two structures of $\mathbb Z[[t]]$-algebra on $\mathbb M$-called left and right and we shall only consider the right one because the left structure is not faithful (it factors through $\mathbb Z_\bullet$).
We can then also write
\[
\mathbb M = \left\{\sum_{k=0}^\infty e_k f_k, f_k \in \mathbb Z[[t]],\quad \mathrm{ord}(f_k) \to \infty \right\}.
\]

\begin{lem}
If $M$ is an $\mathbb M$-module, then
\[
M^{t=0} = M^{e_0=1} = e_0M \simeq \mathbb Z^\mathbb N \otimes_{\mathbb M} M \simeq \mathrm{Hom}_{\mathbb M}(\mathbb Z_\bullet, M).
\]
\end{lem}

\begin{proof}
Since $t't = 1-e_0$, we have $M^{t=0} \subset M^{e_0=1}$.
Moreover, $M^{e_0=1} \subset e_0M$.
The identity $te_0 = 0$ then implies $e_0M \subset M^{t=0}$.
This shows the first two equalities.
The splittings in proposition \eqref{split1} provide (split) exact sequences
\[
0 \to \mathrm{Hom}_\mathbb M(\mathbb Z_\bullet, M) \to \mathrm{Hom}_\mathbb M(\mathbb M, M) \overset {t} \to \mathrm{Hom}_\mathbb M(\mathbb M, M) \to 0
\]
and
\[
0 \to \mathbb Z^\mathbb N \otimes_{\mathbb M} M \to \mathbb M \otimes_{\mathbb M} M \overset {t} \to \mathbb M \otimes_{\mathbb M} M \to 0
\]
showing that $\mathrm{Hom}_\mathbb M(\mathbb Z_\bullet, M) \simeq \mathbb Z^\mathbb N \otimes_{\mathbb M} M \simeq M^{t=0}$.
\end{proof}

\begin{dfn}
The \emph{top row} of an $\mathbb M$-module $M$ is $M(*) :=M^{t=0}$.
\end{dfn}

Note that $M(*)$ is a direct factor (as an abelian group) of $M$.
As an example, we have $\mathbb M(*) \simeq \mathbb Z^\mathbb N$ and $\mathbb Z_\bullet(*) \simeq \mathbb Z$.

There is a tight relation between discrete abelian groups and $\mathbb M$-modules:

\begin{prop} \label{adj}
There exists an adjunction
\[
\mathrm{Hom}_\mathbb M(V_\bullet, M) \simeq \mathrm{Hom}_{\mathbb Z} (V, M(*))
\]
between \emph{discrete} abelian groups and $\mathbb M$-modules.
The functor $V \mapsto V_\bullet$ is fully faithful exact.
The functor $M \mapsto M(*)$ has also a coadjoint.
\end{prop}

\begin{proof}
As a baby example of the technics recalled in the appendix, there exists an adjunction
\[
\mathrm{Hom}_\mathbb M(\mathbb Z_\bullet \otimes_{\mathbb Z} V, M) \simeq \mathrm{Hom}_{\mathbb Z} (V, \mathrm{Hom}_\mathbb M(\mathbb Z_\bullet, M)).
\]
We have $V_{\bullet} \simeq \mathbb Z_\bullet \otimes_{\mathbb Z} V$ and $M(*) \simeq \mathrm{Hom}_\mathbb M(\mathbb Z_\bullet, M)$.
Exactness of $V \mapsto V_\bullet$ follows from the fact that $\mathbb Z_\bullet$ is a free abelian group and fullfaithfulness from the easy identification $V \simeq V_\bullet(*)$.
Finally, the functor $V \mapsto \mathrm{Hom}_{\mathbb M}(\mathbb Z^{\mathbb N}, V)$ provides a coadjoint to $M \mapsto M(*) \simeq \mathbb Z^{\mathbb N} \otimes_{\mathbb M} M$.
\end{proof}

As a consequence of the proposition, the functor $V \mapsto V_\bullet$ preserves all finite limits and all colimits of discrete abelian groups and the functor $M \mapsto M(*)$ preserves all limits and all colimits of $\mathbb M$-modules.
Be careful however that $(\mathbb Z_\bullet)^I \neq (\mathbb Z^I)_\bullet$ when $\mathbb Z^I$ is considered as a discrete abelian group and $I$ is infinite.
We shall remedy to this defect later when we consider a topological enhancement.
Until then, we always mean the former when we remove parenthesis and more generally write $V_\bullet^I := (V_\bullet)^I$.

Recall that we denote by $\mathbb M'$ the ring of row-finite matrices.

\begin{cor} \label{Endis}
There exists canonical ring isomorphisms $\mathrm{End}_\mathbb M(\mathbb M) \simeq \mathbb M'$ and $\mathrm{End}_\mathbb M(\mathbb M^\vee) \simeq \mathbb M$.
\end{cor}

\begin{proof}
We already know that $\mathrm{End}_\mathbb M(\mathbb M) \simeq \mathbb M^{\mathrm{op}} \simeq \mathbb M'$.
Moreover, it follows from proposition \ref{adj}) that
\[
\mathrm{End}_\mathbb M(\mathbb M^\vee) \simeq \mathrm{End}_\mathbb M((\mathbb Z \cdot \mathbb N)_\bullet) \simeq \mathrm{End}_\mathbb Z(\mathbb Z \cdot \mathbb N) \simeq \mathrm{End}_\mathbb Z(\mathbb Z_\bullet) \simeq \mathbb M. \qedhere
\]
\end{proof}

We shall also need later the following result which shows that countable products behave remarkably:

\begin{prop} \label{countlm}
Asumme $I$ is (at most) countable.
Then,
\begin{enumerate}
\item 
there exists an $\mathbb M$-linear isomorphism $\mathbb M^I \simeq \mathbb M$,
\item
there exists a natural isomorphism of abelian groups (for any $J$)
\[
\mathrm{Hom}_{\mathbb M}(\mathbb Z_\bullet^I, \mathbb Z_\bullet^J) \simeq (\mathbb Z \cdot I)^J.
\]
\end{enumerate}
\end{prop}

\begin{proof}
For the first assertion, it is sufficient to remark that $\mathbb M^I \simeq \mathbb Z_\bullet^{\mathbb N \times I} \simeq \mathbb Z_\bullet^{\mathbb N} \simeq \mathbb M$ since $I$ is countable.
The second one quickly reduces to the case $I = \mathbb N$ and $J = 1$ and we have
\[
\mathrm{Hom}_{\mathbb M}(\mathbb Z_\bullet^{\mathbb N}, \mathbb Z_\bullet) \simeq \mathrm{Hom}_{\mathbb M}(\mathbb M, \mathbb Z_\bullet) \simeq \mathbb Z_\bullet \simeq \mathbb Z \cdot \mathbb N. \qedhere
\]
\end{proof}

\section{Monoidal structure on $\mathbb M$-modules} \label{dualsec}

We introduced earlier hypermatrices and we are specially interested in the case $\mathbb M^{\otimes 2}$ of hypermatrices $[c_{ijk}]$ which are finite in $i$ (but not necessarily in $j$ or $k$).
We let $\mathbb M$ act on the left on columns $[c_{* jk}]$ and on the right in two different ways on lines $[c_{i* k}]$ or $[c_{ij*}]$.
Both right structures are exchanged by horizontal transposition turning $[a_{ijk}]$ into $[a_{ikj}]$.
Alternatively, $\mathbb M^{\otimes 2} = \mathbb Z_\bullet^{\mathbb N \times \mathbb N}$ as a left $\mathbb M$-module and we use both (commuting) structures of a right $\mathbb M$-module coming from both identifications $\mathbb Z_\bullet^{\mathbb N \times \mathbb N} \simeq (\mathbb Z_\bullet^{\mathbb N})^{\mathbb N} \simeq \mathbb M^{\mathbb N}$ (using either the first or second factor of $\mathbb N \times \mathbb N$).

Following theorem 3.2 of \cite{Hovey11}, we set:

\begin{dfn}
The \emph{internal tensor product} of two $\mathbb M$-modules $M$ and $N$ is the $\mathbb M$-module
\[
M \otimes_{\mathbb Z_\bullet} N := (\mathbb M^{\otimes 2} \otimes_{\mathbb M} M) \otimes_{\mathbb M} N
\]
obtained by using successively both right structures of $\mathbb M^{\otimes 2}$.
\end{dfn}

\begin{thm}
Internal tensor product endows the category of $\mathbb M$-modules with an additive closed symmetric monoidal structure with $\mathbb Z_\bullet$ as a unit.
\end{thm}

\begin{proof}
Closeness follows from the fact that external tensor product preserves all colimits.
Symmetry follows from the fact that both right structures on $\mathbb M^{\otimes 2}$ are exchanged by an $\mathbb M$-linear automorphism.
Associativity is easily checked and we actually have the general formula
\[
M_1 \otimes_{\mathbb Z_\bullet} \ldots \otimes_{\mathbb Z_\bullet} M_n \simeq \mathbb M^{\otimes n} \otimes_{\mathbb M} (M_1 \otimes_{\mathbb M} \ldots \otimes_{\mathbb M} M_n)
\]
where we use all $n$ right structures on the right hand side.
Finally, since $\mathbb Z_\bullet$ is a finitely presented $\mathbb M$-module, we have (using the first right structure)
\[
\mathbb M^{\otimes 2} \otimes_{\mathbb M} \mathbb Z_\bullet \simeq \mathbb M^{\mathbb N} \otimes_{\mathbb M} \mathbb Z_\bullet \simeq (\mathbb M \otimes_{\mathbb M} \mathbb Z_\bullet)^{\mathbb N} \simeq \mathbb Z_\bullet^{\mathbb N} \simeq \mathbb M
\]
and therefore
\[
\mathbb Z_\bullet \otimes_{\mathbb Z_\bullet} M := (\mathbb M^{\otimes 2} \otimes_{\mathbb M} \mathbb Z_\bullet) \otimes_{\mathbb M} M = \mathbb M \otimes_{\mathbb M} M = M
\]
showing that $\mathbb Z_\bullet$ is the unit.
\end{proof}

It is worth noting that the older notation $\mathbb M^{\otimes n}$ for hypermatrices is (fortunately) compatible with that of internal tensor product.
We shall denote internal Hom (the coadjoint to internal tensor product) as $\mathcal H\mathrm{om}_{\mathbb Z_\bullet}$ (with a calligraphic H) so that
\[
\mathrm{Hom}_{\mathbb M}(M \otimes_{\mathbb Z_\bullet} N, P) \simeq \mathrm{Hom}_{\mathbb M}(M, \mathcal H\mathrm{om}_{\mathbb Z_\bullet} (N, P) )
\]
when $M, N, P$ are $\mathbb M$-modules.
It is then completely formal to check that
\[
\mathrm{Hom}_{\mathbb M}(M, N) \simeq \mathrm{Hom}_{\mathbb M}(\mathbb Z_\bullet, \mathcal H\mathrm{om}_{\mathbb Z_\bullet} (M, N) ) = \mathcal H\mathrm{om}_{\mathbb Z_\bullet} (M, N)(*)
\]
and
\[
\mathcal H\mathrm{om}_{\mathbb Z_\bullet} (M \otimes_{\mathbb Z_\bullet} N, P) \simeq \mathcal H\mathrm{om}_{\mathbb Z_\bullet} (M, \mathcal H\mathrm{om}_{\mathbb Z_\bullet} (N, P) ).
\]

We shall also denote by $\otimes^L_{\mathbb Z_\bullet}$ and $\mathrm R\mathcal H \mathrm {om}_{\mathbb Z_\bullet}$ the derived versions and by $\mathrm{Tor}^{\mathbb Z_\bullet}_q$ and $\mathcal E \mathrm{xt}_{\mathbb Z_\bullet}^q$ their respective cohomology.

\begin{prop}
A projective $\mathbb M$-module $M$ is internally flat and internally projective: if $N$ is any $\mathbb M$-module, then
\[
\forall q > 0,\quad \mathrm{Tor}^{\mathbb Z_\bullet}_q(M, N) = 0\quad \mathrm{and}\quad \mathcal E \mathrm{xt}_{\mathbb Z_\bullet}^q(M, N) = 0.
\]
\end{prop}

\begin{proof}
Follows from the fact that the category of $\mathbb M$-modules has enough projectives.
\end{proof}

\begin{prop} \label{dlhom}
$\mathcal H\mathrm{om}_{\mathbb Z_\bullet} (\mathbb M, \mathbb Z_\bullet) \simeq \mathbb M^\vee$ and $\mathcal H\mathrm{om}_{\mathbb Z_\bullet} (\mathbb M^\vee, \mathbb Z_\bullet) \simeq \mathbb M$.
\end{prop}

\begin{proof}
Given any set $I$,
\[
\mathcal H\mathrm{om}_{\mathbb Z_\bullet} (\mathbb Z_\bullet \cdot I, \mathbb Z_\bullet) \simeq \mathcal H\mathrm{om}_{\mathbb Z_\bullet} (\mathbb Z_\bullet, \mathbb Z_\bullet)^I \simeq \mathcal H\mathrm{om}_{\mathbb Z_\bullet} (\mathbb Z_\bullet, \mathbb Z_\bullet^I) = \mathbb Z_\bullet^I 
\]
so that, in particular, $\mathcal H\mathrm{om}_{\mathbb Z_\bullet} (\mathbb M^\vee, \mathbb Z_\bullet) \simeq \mathbb M$.
By adjunction, this isomorphism provides a morphism $\mathbb M^\vee \to \mathcal H\mathrm{om}_{\mathbb Z_\bullet} (\mathbb M, \mathbb Z_\bullet)$ which is actually bijective:
\begin{align*}
 \mathcal H\mathrm{om}_{\mathbb Z_\bullet} (\mathbb M, \mathbb Z_\bullet) & \simeq \mathrm{Hom}_{\mathbb M}(\mathbb M, \mathcal H\mathrm{om}_{\mathbb Z_\bullet} (\mathbb M, \mathbb Z_\bullet)) 
\\ & \simeq \mathrm{Hom}_{\mathbb M}(\mathbb M \otimes_{\mathbb Z_\bullet} \mathbb M, \mathbb Z_\bullet)
\\& \simeq \mathrm{Hom}_{\mathbb M}(\mathbb Z_\bullet^{\mathbb N \times \mathbb N}, \mathbb Z_\bullet)
\\& \simeq \mathbb Z \cdot (\mathbb N \times \mathbb N)
\\& \simeq \mathbb M^\vee
\end{align*}
by proposition \ref{countlm}.
\end{proof}

As a consequence of the proposition, there is no ambiguity in denoting by
\[
M^\vee := \mathcal H\mathrm{om}_{\mathbb Z_\bullet} (M, \mathbb Z_\bullet)
\]
the internal dual of an $\mathbb M$-module $M$.

\begin{cor} \label{dlfree}
If $I$ is countable\footnote{This is not necessary for the second isomorphism.}, then
\[
(\mathbb Z_\bullet^I)^\vee = \mathbb Z_\bullet\cdot I \quad \mathrm{and} \quad (\mathbb Z_\bullet\cdot I)^\vee = \mathbb Z_\bullet^I.
\]
\end{cor}

\begin{proof}
The case where $I$ is finite as well as the second isomorphism both follow from preservation of limits.
On the other hand, when $I = \mathbb N$, our statement is exactly the content of proposition \ref{dlhom} (up to isomorphism).
\end{proof}

\begin{cor} \label{tensfr}
If $I,J$ are countable, then there exists natural isomorphisms\footnote{$J$ can be uncountable for the second one.}
\[
\mathbb Z_\bullet^I \otimes^{\mathrm L}_{\mathbb Z_\bullet} \mathbb Z_\bullet^J \simeq \mathbb Z_\bullet^{I \times J} \quad \mathrm{and} \quad \mathrm R\mathcal H\mathrm{om}_{\mathbb Z_\bullet} (\mathbb Z_\bullet^I, \mathbb Z_\bullet^J) \simeq (\mathbb Z_\bullet \cdot I)^J.
\]
\end{cor}

\begin{proof}
Since $\mathbb Z_\bullet^I $ is projective, it is sufficient to consider the underived situation.
In the first case, we can assume that $I = J = \mathbb N$ (since the cases $I$ or $J$ finite are essentially trivial) and we have
\[
\mathbb Z_\bullet^{\mathbb N} \otimes_{\mathbb Z_\bullet} \mathbb Z_\bullet^{\mathbb N} \simeq \mathbb M \otimes_{\mathbb Z_\bullet} \mathbb M = \mathbb M^{\otimes 2} = \mathbb Z_\bullet^{\mathbb N \times \mathbb N}
\]
by definition.
In the second case, we may assume $J = 1$ (since internal hom preserves all limits) and then $\mathcal H\mathrm{om}_{\mathbb Z_\bullet} (\mathbb Z_\bullet^I, \mathbb Z_\bullet) = (\mathbb Z_\bullet^I)^\vee =\mathbb Z_\bullet \cdot I$ thanks to corollary \ref{dlfree}.
\end{proof}

As a consequence, we more generally have $\mathbb Z_\bullet^I \otimes^{\mathrm L}_{\mathbb Z_\bullet} M \simeq M^I$ when $M$ is a finitely presented $\mathbb M$-module and $I$ is countable.

\begin{lem}
If $V$ is a \emph{discrete} abelian group and $M$ is an $\mathbb M$-module then
\[
\mathcal H \mathrm{om}_{\mathbb Z_\bullet}(V_\bullet, M) \simeq \mathrm{Hom}_{\mathbb Z}(V, M) \quad \mathrm{and} \quad V_\bullet \otimes_{\mathbb Z_\bullet} M \simeq V\otimes_{\mathbb Z} M.
\]
\end{lem}

\begin{proof}
We may assume that $V = \mathbb Z$ and there is nothing to do.
\end{proof}

\begin{prop} \label{monadj}
The adjunction $V \mapsto V_\bullet$ and $M \mapsto M(*)$ between \emph{discrete} abelian groups and $\mathbb M$-modules is monoidal:
\[
V_\bullet \otimes_{\mathbb Z_\bullet} W_\bullet \simeq (V\otimes_{\mathbb Z} W)_\bullet \quad \mathrm{and} \quad M(*) \otimes_{\mathbb Z} N(*) \to (M \otimes_{\mathbb Z_\bullet} N)(*)
\]
\end{prop}

\begin{proof}
Since all the functors involved preserve all colimits, the first isomorphism reduces to the trivial case $\mathbb Z_\bullet \otimes_{\mathbb Z_\bullet} \mathbb Z_\bullet \simeq (\mathbb Z\otimes_{\mathbb Z} \mathbb Z)_\bullet$.
The second one formally follows by adjunction.
\end{proof}

Note that this adjunction is not strong monoidal because already when $M=N =\mathbb M$, we get $\mathbb Z^{\mathbb N} \otimes_{\mathbb Z} \mathbb Z^{\mathbb N} \not\simeq \mathbb Z^{\mathbb N \times \mathbb N}$.
The explicit formula for the first isomorphism sounds very natural:
\[
\begin{tikzcd}[row sep=tiny]
\mathbb M^{\otimes 2} \otimes_{\mathbb M} V_\bullet \otimes_{\mathbb M} W_\bullet = V_\bullet \otimes_{\mathbb Z_\bullet} W_\bullet \ar[r] & (V \otimes_{\mathbb Z} W)_\bullet
\\
\left[ c_{k,n,m}\right]_{k,n,m \in \mathbb N}\otimes (x_n)_{n \in \mathbb N} \otimes (y_m)_{m \in \mathbb N} \ar[r, maps to] & \left(\sum_{n,m} c_{k,n,m} x_n \otimes y_m\right)_{k \in \mathbb N}.
\end{tikzcd}
\]

Finally, it is important to mention that there is \emph{no} natural bilinear map $M \times N \to M \otimes_{\mathbb Z_\bullet} N$ in general and the next result is the best we can hope for in this direction:

\begin{prop} \label{bilmon}
If $M$ and $N$ are two $\mathbb M$-modules, then there exists a natural bilinear map $M(*) \times N \to M \otimes_{\mathbb Z_\bullet} N$.
\end{prop}

\begin{proof}
The unit map $M \to \mathcal H \mathrm{om}_{\mathbb Z_\bullet}(N, M \otimes_{\mathbb Z_\bullet} N)$ provides
\[
M(*) \to \mathcal H \mathrm{om}_{\mathbb Z_\bullet}(N, M \otimes_{\mathbb Z_\bullet} N)(*) = \mathrm{Hom}_{\mathbb M}(N, M \otimes_{\mathbb Z_\bullet} N)
\]
and our bilinear map is obtained by adjunction.
\end{proof}

\section{Finitely presented $\mathbb M$-modules}

\begin{lem} \label{fptr}
For an $\mathbb M$-module $M$, the following are equivalent:
\begin{enumerate}
\item $M$ is finitely presented,
\item There exists $u \in \mathbb M$ such that $M \simeq \mathbb M/\mathbb Mu$,
\item There exists a right exact sequence $\mathbb Z_\bullet^I \to \mathbb Z_\bullet^J \to M \to 0$ with $I,J$ countable.
\end{enumerate}
\end{lem}

\begin{proof}
($(1) \Leftrightarrow (2)$)
By definition, an $\mathbb M$-module $M$ is finitely presented if and only if there exists a right exact sequence
\[
\mathbb M^m \to \mathbb M^n \to M \to 0.
\]
We may assume $m=n=1$ thanks to proposition \ref{countlm} and any endomorphism of $\mathbb M$ is necessarily right multiplication by some $u \in \mathbb M$.

($(2) \Rightarrow (3)$)
We can use right multiplication by $u$ and the isomorphism $\mathbb M \simeq (\mathbb Z_\bullet)^\mathbb N$.

($(3) \Rightarrow (2)$)
Possibly after taking product with the identity of $\mathbb Z^{\mathbb N}$, we may assume that $I$ and $J$ are infinite but then $\mathbb Z_\bullet^I \simeq \mathbb Z_\bullet^J \simeq \mathbb M$.
\end{proof}

Note also that a finitely generated $\mathbb M$-module (resp.\ ideal) is automatically monogenic (resp.\ principal) since $\mathbb M^n \simeq \mathbb M$.

\begin{lem}
For an $\mathbb M$-module $M$, the following are equivalent:
\begin{enumerate}
\item $M$ is finitely generated\footnote{Or, equivalenlty, presented.} projective,
\item $M$ is (isomorphic to) a finitely generated projective ideal of $\mathbb M$,
\item $M \simeq \mathbb Me$ with $e \in \mathbb M$ idempotent,
\item $M \simeq \mathbb Z_\bullet^I$ with $I$ countable.
\end{enumerate}
\end{lem}

\begin{proof}
($(1) \Leftrightarrow (2)$) A finitely generated $\mathbb M$-module is projective if and only if it is a direct factor of some $\mathbb M^n \simeq \mathbb M$.

($(1) \Leftrightarrow (3)$) An $\mathbb M$-module $M$ is finitely generated projective if and only if there exists a projector $e : \mathbb M^n \to \mathbb M^n$ (i.e. $e \circ e = e$) such that $M \simeq \mathrm{im}\; e$.
Here, we have $\mathbb M^n \simeq \mathbb M$.

($(2) \Rightarrow (4)$)
We use here proposition \ref{dlfree}.
Since $M$ is a direct factor of $\mathbb M$ and $\mathbb M = \mathbb M^{\vee\vee}$, we shall also have $M = M^{\vee\vee}$.
Moreover, $M^\vee$ will be a direct factor of $\mathbb M^\vee = (\mathbb Z \cdot \mathbb N)_\bullet$.
It follows that there exists a direct factor $F$ of $\mathbb Z \cdot \mathbb N$ such that $M^\vee = F_\bullet$ (this is formal since $-_\bullet$ is fully faithful and admits a coadjoint).
We can then write $F =: \mathbb Z\cdot I$ with $I$ countable and we shall then have $M^\vee = (\mathbb Z\cdot I)_\bullet$ and finally $M = M^{\vee\vee} = \mathbb Z_\bullet^I$.

($(4) \Rightarrow (1)$)
Either $M \simeq \mathbb M$ or $M \simeq \mathbb Z_\bullet^n$ for some $n \in \mathbb N$, but $\mathbb Z_\bullet$ itself is finitely generated projective.
\end{proof}

\begin{cor}
The ideal $\mathbb M^\vee$ of $\mathbb M$ is projective but not finitely generated.
\end{cor}

\begin{proof}
It is projective as a direct sum of projectives.
If it were finitely presented, there would exist an isomorphism $\mathbb M^\vee \simeq \mathbb M$ or $\mathbb M^\vee \simeq \mathbb Z_\bullet^n$.
By adjunction, we would then have $\mathbb Z \cdot \mathbb N \simeq \mathbb Z^\mathbb N$, but the former is countable and the later is not, or $\mathbb Z^n \cdot \mathbb N \simeq \mathbb Z^n$ which is even worse (infinite rank and finite rank).
\end{proof}

\begin{prop}
Finitely presented $\mathbb M$-modules are stable under finite colimits, \emph{countable} limits, extensions and internal tensor product.
Any $\mathbb M$-module is a filtered colimit of finitely presented $\mathbb M$-modules.
\end{prop}

\begin{proof}
Only stability under countable limits and internal tensor product require a specific argument.
Countable limits reduce to the case of countable products.
Since (even infinite) products are exact, it reduces to showing that $\mathbb M^I$ is finitely presented when $I$ is countable but then $\mathbb M^I \simeq \mathbb M$.
The case of tensor products reduces by right exactness to corollary \ref{tensfr}.
\end{proof}

Note that finitely presented $\mathbb M$-modules are \emph{not} stable under internal Hom however.

\begin{lem} \label{split}
Any finitely presented $\mathbb M$-module $M$ splits as
\[
M \simeq \mathbb Z_\bullet^E \oplus \mathcal E\mathrm{xt}_{\mathbb Z_\bullet}(T_\bullet, \mathbb Z_\bullet)
\]
with a countable set $E$ and a countable discrete abelian group $T$ such that\footnote{This is the naive dual $T^\vee := \mathrm{Hom}_{\mathbb Z}(T, \mathbb Z)$.} $T^\vee = 0$.
\end{lem}

\begin{proof} (After Clausen-Scholze)
There exists a presentation $\mathbb Z_\bullet^I \overset f \to \mathbb Z_\bullet^J \to M \to 0$ with $I,J$ countable and, by duality, a left exact sequence $0 \to M^\vee \to (\mathbb Z \cdot J)_\bullet \overset {f^\vee} \to (\mathbb Z\cdot I)_\bullet $.
We can write $f^\vee = g_\bullet$ with $g : \mathbb Z \cdot J \to \mathbb Z\cdot I$.
Then, there exists a split exact sequence $0 \to \mathbb Z \cdot E \to \mathbb Z \cdot J \to \mathbb Z\cdot K \to 0$ with $\ker g \simeq \mathbb Z\cdot E$ and $\mathrm{im}\; g \simeq \mathbb Z \cdot K$.
In particular, $g$ splits as $g : \mathbb Z \cdot J \twoheadrightarrow \mathbb Z\cdot K \rightarrowtail \mathbb Z\cdot I$.
By duality, there exists a split exact sequence $0 \to \mathbb Z_\bullet^K \to \mathbb Z_\bullet^J \to \mathbb Z_\bullet^E \to 0$ as well as a splitting $f : \mathbb Z_\bullet^I \to \mathbb Z_\bullet^K \rightarrowtail \mathbb Z_\bullet^J$.
We can then fill-up the commutative diagram with exact rows
\[
\begin{tikzcd}
& \mathbb Z_\bullet^I \ar[r, " f " ] \ar[d] & \mathbb Z_\bullet^J \ar[r] \ar[d, equal] & M \ar[r] \ar[d, dashed, twoheadrightarrow] & 0
\\
0 \ar[r] & \mathbb Z_\bullet^K \ar[r] & \mathbb Z_\bullet^I \ar[r] & \mathbb Z_\bullet^E \ar[r] & 0
\end{tikzcd}
\]
and obtain a splitting $M = \mathbb Z_\bullet^E \oplus N$ (since $\mathbb Z_\bullet^E$ is projective).
Since $M^\vee = \ker f^\vee = \ker g_\bullet = \mathbb Z_\bullet \cdot E$, then, necessarily, $N^\vee = 0$.
After replacing $M$ by $N$, we may therefore assume that $M^\vee = 0$ so that $g$ is now injective.
Then, there exists a short exact sequence $0 \to \mathbb Z \cdot J \overset g \to \mathbb Z\cdot I \to Q \to 0$ with $Q$ countable providing a (long) exact sequence $0 \to Q_\bullet^\vee \to \mathbb Z_\bullet^I \overset f \to \mathbb Z_\bullet^J \to \mathcal E\mathrm{xt}_{\mathbb Z_\bullet}(Q_\bullet, \mathbb Z_\bullet) \to 0$ and in particular an identification $M \simeq \mathcal E\mathrm{xt}_{\mathbb Z_\bullet}(Q_\bullet, \mathbb Z_\bullet)$.
Finally, since $Q$ is countable, we know from corollary 19.3 in \cite{Fuchs70} that $Q = F \oplus T$ with $F$ free and $T^\vee = 0$ so that $\mathcal E\mathrm{xt}_{\mathbb Z_\bullet}(Q_\bullet, \mathbb Z_\bullet) = \mathcal E\mathrm{xt}_{\mathbb Z_\bullet}(T_\bullet, \mathbb Z_\bullet)$.
\end{proof}

\begin{thm} \label{cohM}
A finitely presented $\mathbb M$-module $M$ has projective dimension at most one.
\end{thm}

\begin{proof}
Thanks to lemma \ref{split}, we may assume that $M = \mathcal E\mathrm{xt}_{\mathbb Z_\bullet}(T_\bullet, \mathbb Z_\bullet)$ with $T$ countable and $T^\vee = 0$.
Then, there exists a short exact sequence $0 \to \mathbb Z \cdot I \to \mathbb Z \cdot J \to T \to 0$ with $I,J$ countable.
Thanks to corollary \ref{torlm} below (discrete case), we obtain a short exact sequence
\[
0 \to \mathbb Z_\bullet^I \to \mathbb Z_\bullet^J \to \mathcal E\mathrm{xt}_{\mathbb Z_\bullet}(T_\bullet, \mathbb Z_\bullet) \to 0. \qedhere
\]
\end{proof}

The following statement is ``equivalent'' to the previous theorem:

\begin{cor} \label{fgpr}
Any finitely generated ideal of $\mathbb M$ is projective. \qed
\end{cor}


The next corollary is very likely well known but I was unable to find the statement in the literature:

\begin{cor}
The ring $\mathbb M$ is (left) coherent. \qed
\end{cor}


\section{Digressions on topological abelian groups}

This section is totally independent of the previous ones.
We gather here several results about topological abelian groups and more generally topological rings and modules that will be used later.
Concerning linear topologies, we send the reader to chapter 1 of \cite{Positselski24}, section 2.1 of \cite{Baldassarri25} or \cite{MauroPombo17}.

Let us first briefly recall the behavior of limits and colimits in this situation.
The forgetful functor from topological spaces to sets has both an \emph{exact} adjoint (discretization) and a coadjoint (indiscretization).
It formally follows that the same holds for the forgetful functors from topological rings to plain rings or from topological $A$-modules to plain $A$-modules when $A$ is a topological ring.
In particular, these forgetful functors preserve all limits and all colimits.
Now, the forgetful functor from $A$-modules to sets has an adjoint (freeness) but no coadjoint: it preserves all limits, all filtered colimits and cokernels (aka coequalizer) but \emph{not} coproducts ($\cup \neq +$).
Again, the same holds for the forgetful functor from topological $A$-modules to topological spaces.
In particular, the topology of a finite direct sum is the product topology but the topology of an infinite direct sum is \emph{not} induced by the product topology.

\begin{dfn} \label{linab}
If $R$ is a topological ring, then an \emph{$R$-linear\footnote{Or more precisely \emph{$R$-linearly topologized}.} ring} $A$ (resp.\ \emph{$R$-linear $A$-module} $M$) is a topological $R$-algebra $A$ (resp.\ topological $A$-module $M$) such that open $R$-submodules form a basis of neighborhoods of zero.
\end{dfn}

We shall simply say $\emph{linear}$ in the case $R = \mathbb Z$ and call in particular \emph{linear abelian group} a topological abelian group whose topology is linear.

\begin{xmps}
\begin{enumerate}
\item The discrete topology is linear.
\item The topology of an adic ring $A$ as well as any $A$-module $M$ (endowed with its adic topology) is $A$-linear.
\item The topology of a Huber ring or, more generally, a finitely presented $A$-module (endowed with its canonical topology) is linear (meaning $\mathbb Z$-linear) but \emph{not} $A$-linear in general.
\end{enumerate}
\end{xmps}

$R$-linear $A$-modules form a reflective and coreflective subcategory of the category of all topological $A$-modules.
The reflection (adjoint to inclusion) $V \mapsto V^{\mathrm{lin}}$ is obtained by endowing the underlying module with the coarser topology having open $R$-submodules as a basis of neighborhoods of zero but the coreflection (coadjoint to inclusion) is more mysterious. Nevertheless, the inclusion functor preserves all limits and all colimits.


If $M$ is an $R$-linear $A$-module, then the \emph{(Hausdorff) completion} of $M$ is the $R$-linear $A$-module
\begin{equation} \label{cmpdef}
\widehat M := \varprojlim_{M'} M/M',
\end{equation}
where $M'$ runs through all (or a basis of) open $R$-submodules, and $M$ is said to be \emph{complete\footnote{There exists a more general notion of completeness for arbitrary topological abelian groups (any Cauchy filter is convergent) that we shall not need.} (Hausdorff)} if the canonical map is an isomorphism $M \simeq \widehat M$.
Completion is a reflection and any limit of complete $R$-linear $A$-modules is therefore automatically complete. 
This is actually also the case for (infinite) direct sums (endowed with the direct sum topology) - but not other colimits in general.
When $R = A$, this notion of $A$-linear topology is closely related to that of pro-$A$-modules: if $M,N$ are two $A$-linear $A$-modules, then
\[
\mathrm{Hom}_{A\mhyphen\mathrm{cont}}(M, \widehat N) =\varprojlim_{N'} \varinjlim_{M'} \mathrm{Hom}_{A}(M/M', N/N')
\]
when $M'$ and $N'$ run through open $A$-submodules.
In particular, this $\mathrm{Hom}$ is naturally a complete linear abelian group\footnote{But we shall favor in practice compact-open topology.} (and even an $A$-linear $A$-module when $A$ is commutative).
In the case $R$ is commutative, there exists a universal $R$-linear topology on the tensor product of two $R$-linear (right and left) $A$-modules.
In the case $A$ too is commutative, then (complete) $R$-linear $A$-modules form an additive symmetric monoidal category with respect to (completed) tensor product (which is \emph{not} closed).

By composition, the category of complete $R$-linear $A$-modules is a reflective subcategory of the category of all topological $A$-modules whose reflection $M \mapsto M^{(\wedge)} := M^{\mathrm{lin}, \wedge}$ may be called \emph{linear completion} and is given by the same formula as in \eqref{cmpdef}.
In other words, this is the inverse limit of the discrete quotients.
In particular, linear completion coincides with Hausdorff completion when $M$ is an $R$-linear $A$-module (but \emph{not} otherwise in general).

An $R$-linear $A$-module $M$ is metrizable if and only if it is first countable Hausdorff if and only if there exists a countable basis of open neighborhoods $M_n$ such that $\bigcap M_n = 0$.
Completion preserves quotients by metrizable $R$-linear $A$-modules (but not all finite colimits).
On the bad side, a countable direct sum of metrizable $R$-linear $A$-modules is usually \emph{not} metrizable itself (with respect to the direct sum topology).

As indicated earlier, if $A$ is a topological ring, then the forgetful functor from topological $A$-modules to plain $A$-modules admits an exact adjoint.
Concretely an $A$-module $M$ will be endowed with the \emph{adjoint} or \emph{colimit} topology which is the finest $A$-module topology. 
Then, by definition, if $N$ is a topological $A$-module, any $A$-linear map $M \to N$ will be continuous.
In practice, if $A \cdot I \twoheadrightarrow M$ is a surjective $A$-linear map from a free $A$-module, then the adjoint topology on $A \cdot I$ (resp.\ $M$) is the direct sum topology (resp.\ the quotient topology) and this does not depend on the choices.
Be careful however that the topology induced on a submodule by the adjoint topology is usually coarser than the adjoint topology itself.
Finally, when $A$ is an $R$-linear ring, then the adjoint topology on an $A$-module $M$ is automatically $R$-linear.
All this sounds quite appealing but we should not rely solely on our intuition when using adjoint topology: for example, the polynomial ring $\mathbb Z_p[X]$ is already complete for the adjoint topology when $\mathbb Z_p$ is equipped with the $p$-adic topology.

\begin{dfn}
If $V$ is a \emph{topological} abelian group, then the \emph{set of null sequences} of $V$ is
\[
V_\bullet:= \{(s_n)_{n \in \mathbb N}, \quad s_n \to 0\}
\]
 (seen as column vectors).
\end{dfn}

Of course, when $V$ is discrete, we recover our older notation from section \ref{secmain}.
This applies in particular to $\mathbb Z$ which is (unless otherwise specified) always endowed with the discrete topology.
Be careful however that $\mathbb Z^I$ will be endowed with the \emph{product} topology which is not discrete when $I$ is infinite.

If $V$ is any topological abelian group, we shall then (still) denote by $t$ the endomorphism of $V_\bullet$ that shifts components up, so that
\[
t(s_0, s_1, s_2, \ldots) = (s_1, s_2, s_3, \ldots).
\]
We then get a split exact sequence $0 \to V \to V_\bullet \overset {t} \to V_\bullet \to 0$ and in particular an isomorphism $V_\bullet(*) := V_\bullet^{t=0} \simeq V$.
The map $1-t$ is injective (resp.\ bijective) on $V_\bullet$ if and only if $V$ is Hausdorff (resp.\ any null sequence is uniquely summable).
Note that, there also exists an endomorphism $e_k$ for each $k \in \mathbb N$ given by
\[
e_k(s_0, s_1, s_2, \ldots) = (0, \ldots, 0, s_0, 0, \ldots)
\]
with $s_0$ at the $k$th rank.
Together with $t$, they satisfy commutation rules \eqref{rul1} and \eqref{rul2}.

If $V$ is a topological abelian group and we want to see $V_\bullet$ itself as a topological abelian group, we will then choose as a basis of neighborhoods of zero the subsets $U_\bullet$ of null sequences in a neighborhood $U$ of zero in $V$.

Recall that we equip $\overline {\mathbb N} := \mathbb N \cup \{\infty\}$ with its profinite topology.
Then, there exists a split exact sequence of topological abelian groups
\begin{equation} \label{cotop}
0 \longrightarrow V_\bullet \longrightarrow \mathrm{Hom}_{\mathrm{cont}}(\overline{\mathbb N}, V) \overset \infty \longrightarrow V \longrightarrow 0
\end{equation}
when the middle term is endowed with the (compact-) open topology (so that the subsets $\mathrm{Hom}_{\mathrm{cont}}(\overline{\mathbb N}, U)$ form a basis of neighborhoods of zero).
Note that the shift $t$ on $V_\bullet$ is induced by the shift map $n \mapsto n+1$ on $\overline{\mathbb N}$.
This presentation shows that the functor of $V \mapsto V_\bullet$ preserves all limits of topological abelian groups.
In particular, there is no ambiguity in writing $V_\bullet^I$ as long as we take the topology into account.

If $S$ is topological space (resp.\ and $\infty \in S$ is a given point), then there exists a topological abelian group $\mathbb Z \cdot S$ (resp.\ $\mathbb Z \cdot S^*$) which is universal for morphisms (resp.\ for pointed morphisms) to topological abelian groups.
This is discussed in \cite{Morris74}) where both flavors are called after Markov\footnote{A. A. Markov, son of A. Markov.} and Graev respectively.
When $S$ is completely regular\footnote{Points and closed sets can be separated by continuous real valued functions.}, then the underlying abelian group is the usual (pointed) free abelian group 
\[
\mathbb Z \cdot S^{(*)} = \left\{ \sum_{\mathrm{finite}} c_s e_s, \quad c_s \in \mathbb Z, \quad s \in S^{(*)} \right\}
\]
- and conversely.
Moreover, the topology of $\mathbb Z \cdot S^{(*)}$ is the finest abelian group topology making the inclusion of $S$ continuous.
Finally, $\mathbb Z \cdot S \simeq \mathbb Z \cdot S^* \oplus \mathbb Z \cdot \infty$ as topological abelian groups.
When $S$ is compact metrizable, then the topology of $\mathbb Z \cdot S^{(*)}$ is sequential\footnote{Closed/open/continuity may be characterized by convergent sequences -- equivalently, this is a quotient of a metrizable, or a quotient of first countable space.} Hausdorff.
Be careful however that it is \emph{never} locally compact nor metrizable unless $S$ is discrete.
We may always consider the Archimedean norm $\|\sum c_ss\|_1= \sum_s |c_s|$ on $\mathbb Z \cdot S$ and set 
\[
(\mathbb Z \cdot S)_N = \left\{x \in \mathbb Z \cdot S,\quad \|x\|_1 \leq N \right\}.
\]
When $S$ is compact Hausdorff, we can then write $\mathbb Z \cdot S := \varinjlim_{N \in \mathbb N} (\mathbb Z \cdot S)_N$ as a filtered colimit of compact Hausdorff subspaces.
In the case $S$ is a Stone (aka profinite) space, then $S := \varprojlim_{i \in I} S_i$ is a (filtered) limit of finite discrete sets (with surjective transition maps), and $(\mathbb Z \cdot S)_N = \varprojlim_{i \in \mathbb N} (\mathbb Z \cdot S_i)_N$ also is a Stone space.

We shall denote by $V^\vee := \mathrm{Hom}_{\mathbb Z-\mathrm{cont}}(V, \mathbb Z)$ the topological dual of a topological abelian group $V$ (equipped with the compact-open topology).
Recall that the category of totally disconnected topological spaces is a reflective subcategory of the category of all topological spaces with reflection $S \mapsto \pi_0(S)$ (when $\pi_0(S)$ is equipped with the quotient topology).
Since $\mathbb Z$ has the discrete (and in particular totally disconnected) topology, we have $(\mathbb Z \cdot S)^\vee \simeq (\mathbb Z \cdot \pi_0(S))^\vee$.
In particular, the case of a compact Hausdorff space $S$ reduces to that of the Stone space $\pi_0(S)$.

Recall that we also introduced above the linear completion $M^{(\wedge)}$ of an arbitrary topological abelian group $M$ as the inverse limit of its discrete quotients.

\begin{prop} \label{ddcomp}
If $S$ is a Stone space and we let $S'$ run through all finite quotients of $S$, then
\begin{enumerate}
\item
$
(\mathbb Z \cdot S)^{(\wedge)} \simeq \varprojlim_{S \twoheadrightarrow S'} \mathbb Z \cdot S',
$
\item
$(\mathbb Z \cdot S)^\vee \simeq \varinjlim_{S \twoheadrightarrow S'} (\mathbb Z \cdot S')^\vee$,
\item
$
(\mathbb Z \cdot S)^{(\wedge)} \simeq (\mathbb Z\cdot S)^{\vee\vee}.
$
\end{enumerate}
\end{prop}

\begin{proof}
Notice first that the right hand side is indeed a complete linear abelian group.
Now, we have to show that, if $V$ is another complete linear abelian group, then
\[
\mathrm{Hom}_{\mathbb Z-\mathrm{cont}}\left( \varprojlim_{S \twoheadrightarrow S'} \mathbb Z \cdot S', V\right) \simeq \mathrm{Hom}_{\mathbb Z-\mathrm{cont}}(\mathbb Z \cdot S, V).
\]
We may assume that $V$ is discrete in which case
\begin{align*}
\mathrm{Hom}_{\mathbb Z-\mathrm{cont}}\left( \varprojlim_{S \twoheadrightarrow S'} \mathbb Z \cdot S', V\right) &\simeq \varinjlim_{S \twoheadrightarrow S'} \mathrm{Hom}_{\mathbb Z}(\mathbb Z \cdot S', V)
\\ &\simeq \varinjlim_{S \twoheadrightarrow S'} \mathrm{Hom}(S', V)
\\ &\simeq \mathrm{Hom}_{\mathrm{cont}}(S, V)
\\ &\simeq \mathrm{Hom}_{\mathbb Z-\mathrm{cont}}(\mathbb Z \cdot S, V).
\end{align*}

The last lines in the case $V = \mathbb Z$ also provide
\[
(\mathbb Z\cdot S)^\vee \simeq \varinjlim_{S \twoheadrightarrow S'} (\mathbb Z \cdot S')^\vee.
\]

Finally, when $S$ is finite, then $\mathbb Z \cdot S \simeq (\mathbb Z\cdot S)^{\vee\vee}$ and, in general,
\begin{align*}
(\mathbb Z\cdot S)^{\vee\vee} &\simeq \mathrm{Hom}_{\mathbb Z}\left(\varinjlim_{S \twoheadrightarrow S'} (\mathbb Z \cdot S')^\vee, \mathbb Z\right)
\\ &\simeq \varprojlim_{S \twoheadrightarrow S'} \mathrm{Hom}_{\mathbb Z}((\mathbb Z \cdot S')^\vee, \mathbb Z)
\\ &\simeq \varprojlim_{S \twoheadrightarrow S'} (\mathbb Z\cdot S')^{\vee\vee}
\\ &\simeq \varprojlim_{S \twoheadrightarrow S'} \mathbb Z \cdot S'
\\ &\simeq (\mathbb Z \cdot S)^{(\wedge)}. \qedhere
\end{align*}
\end{proof}

We may also notice that the topology of $(\mathbb Z \cdot S)^\vee$ is discrete: since $S$ is compact and $0$ is open in $\mathbb Z$, the set of all continuous maps from $\mathbb Z \cdot S$ to $\mathbb Z$ sending $S$ to $0$ is open (for the compact open topology) and this set is reduced to the zero map.
A lot more is true:

\begin{thm}[Nöbeling]\label{Nobel}
If $S$ is a compact Hausdorff (metrizable) topological space, then $(\mathbb Z\cdot S)^\vee$ is a \emph{free} discrete abelian group (of countable rank).
\end{thm}

\begin{proof}
This is shown in \cite{Nobeling68} with a more recent proof in \cite{Asgeirsson24b} (but see also propostion 2.1.7 of \cite{Camargo26} in the metrizable case).
\end{proof}

Nöbeling's theorem states that, if $S$ is a compact Hausdorff space, then there exists a set $I$ and a (highly non-natural) isomorphism $(\mathbb Z\cdot S)^\vee \simeq \mathbb Z \cdot I$ (with the discrete topology).
From this, we deduce an isomorphism $(\mathbb Z\cdot S)^{\vee\vee} \simeq \mathbb Z^I$ for the topological bidual.

We shall consider more specifically the case $S = \overline {\mathbb N}$ pointed at $\infty$ and denote by $P := \mathbb Z \cdot \overline {\mathbb N}^*$ the \emph{Graev free abelian group of countable rank}.
Thus, $P$ is a free abelian group on $(e_n)_{n \in\mathbb N}$ which is not metrizable and not locally compact but may be written as a filtered colimit of metrizable Stone subspaces.
Recall also that $\mathbb Z \cdot \overline {\mathbb N} \simeq P \oplus \mathbb Z \cdot \infty$.

\begin{lem} \label{Pbul}
If $V$ is a topological abelian group, then $V_\bullet \simeq \mathrm{Hom}_{\mathbb Z\mhyphen\mathrm{cont}}(P, V)$.
\end{lem}

\begin{proof}
Compare the split exact sequence
\[
0 \to \mathrm{Hom}_{\mathbb Z\mhyphen\mathrm{cont}}(P, V) \to \mathrm{Hom}_{\mathbb Z\mhyphen\mathrm{cont}}(\mathbb Z \cdot \overline{\mathbb N}, V) \to \mathrm{Hom}_{\mathbb Z\mhyphen\mathrm{cont}}(\mathbb Z \cdot \infty, V) \to 0
\]
with \eqref{cotop}.
\end{proof}

As a consequence of the lemma, we see that $P^\vee \simeq \mathbb Z_\bullet$, so that $P^{\vee\vee} \simeq \mathbb Z^\mathbb N$ and the bidual map $P \rightarrowtail P^{\vee\vee}$ identifies with the obvious inclusion map $P \rightarrowtail \mathbb Z^\mathbb N$.

Recall also that the pointed continuous map
\[
t : \overline {\mathbb N} \to \overline {\mathbb N}, \quad n \mapsto n+1
\]
induces a endomorphism $t : P \to P$ that recovers our endomorphism $t$ of $V_\bullet = \mathrm{Hom}_{\mathbb Z\mhyphen\mathrm{cont}}(P, V)$.
Odd enough, there exists therefore a split exact sequence $0 \to P \overset t \to P \to \mathbb Z \to 0$ (evaluation at $0$) so that $P \simeq \mathbb Z \oplus P \simeq P \oplus \mathbb Z \simeq \mathbb Z \cdot \overline{\mathbb N}$ as topological abelian groups (Graev = Markov).
Also worth mentioning, the pointed continuous map
\[
e_k : \overline{\mathbb N} \to \overline{\mathbb N}, \quad \left\{\begin{array}{ll} 0 \mapsto k \\ n \mapsto \infty\ \mathrm{otherwise}, \end{array}\right.
\]
defines an endomorphism of $P$ that induces the above endomorphism $e_k$ of $V_\bullet$.


\section{$\mathbb M$-modules and linear abelian groups} \label{Seclin}

In a complete linear abelian group $V$, any null sequence is uniquely summable or, in other words, a series whose general term goes to zero is automatically (uniquely) convergent\footnote{This is the dream of freshpeople discovering calculus.}.
There exists therefore a natural action
\[
[c_{kn}] \times (s_{k}) = \left(\sum_n c_{kn}s_{n}\right).
\]
of $\mathbb M$ on the set $V_\bullet$ of null sequences generalizing the case of a discrete $V$.

\begin{xmps}
\begin{enumerate}
\item
We have $\mathbb Z[[t]]_\bullet \simeq \mathbb M$ ($t$-adic topology) and $\mathbb Z[t]_\bullet \simeq \mathbb M^\vee$ (discrete topology).
More precisely, we identify $(f_n)$ with $\sum e_nf_n$ in both cases.
\item
We have (discrete topology) $\mathbb Z[1/n]_\bullet \simeq \mathbb M^\vee/\mathbb M^\vee(1-nt)$ with 
\[
1-nt = \left[\begin{array}{ccccc} 1 & -n & 0 & \cdots \\ 0 & 1 & -n & \ddots \\ \vdots & \ddots & \ddots & \ddots \end{array} \right].
\]
\item
We have ($p$-adic topology) $\mathbb Z_{p\bullet} \simeq \mathbb M/\mathbb M(t-p)$ with 
\[
t-p = \left[\begin{array}{ccccc} -p & 1 & 0 & \cdots \\ 0 & -p & 1 & \ddots \\ \vdots & \ddots & \ddots & \ddots \end{array} \right].
\]
\item We have ($p$-adic topology) $\mathbb Q_{p\bullet} \simeq \mathbb Z_{p\bullet}[1/p] \simeq \varinjlim_{p\cdot} \mathbb Z_{p\bullet}$.
\item Let
\[
 \mathbb Q_p\{X\} := \left\{\sum a_n X^n, a_n \to 0\right\} \subset \mathbb Q_p[[X]] \quad \mathrm{and} \quad A := \varinjlim_{X \mapsto pX} \mathbb Q_p\{X\}_\bullet
\]
(with respect to the $p$-adic topology on $\mathbb Q_p\{X\}$).
Then $A$ is \emph{not} complete.
However, $A$ is a genuine $\mathbb M$-module sitting in between $\mathbb Q_p\{X\}_\bullet$ and $\mathbb Q_p[[X]]_\bullet$ (it represents the local ring at the origin).
\end{enumerate}
\end{xmps}

\begin{lem} \label{endcont}
The matrix ring $\mathbb M$ is a complete metrizable $\mathbb M$-linear \emph{ring} (for the canonical topology).
\end{lem}

\begin{proof}
We know from proposition \ref{tMod} that the canonical topology is identical to the $t$-adic topology and that $\mathbb M$ is then complete as a topological $\mathbb M$-module.
Hence, it only remains to check that any right multiplication by some $u = [a_{kn}] \in \mathbb M$ is continuous.
Since all columns of $u$ are finite, if $r \in \mathbb N$, then there exists $m \in \mathbb N$ such that $a_{kn} = 0$ for $n < r$ and $k \geq m$.
It follows that $t^mu = t^m \sum a_{kn}e_kt^n = \sum a_{kn}e_{k-m}t^n\in \mathbb Mt^r$.
\end{proof}

Unless otherwise specified, an $\mathbb M$-module $M$ will be endowed with the \emph{adjoint} topology (the finest $\mathbb M$-module topology\footnote{As discussed in the previous section.}) that we shall call \emph{canonical}.
We shall also call \emph{canonical} the topology on the abelian group $M(*)$ induced by the canonical topology of $M$ (or equivelently the quotient topology since $M(*)$ is a direct factor).
Recall also\footnote{As discussed in the previous section :-).} that when $V$ is a complete linear abelian group, $V_\bullet$ inherits a topology from $V$.
This turns $V_\bullet$ into a complete $\mathbb M$-linear $\mathbb M$-module: the $\mathbb M$-submodules $V'_\bullet$, when $V'$ runs trough all open subgroups of $V$, form a basis of neighborhoods of zero.
Be careful that the topology on $V_\bullet$ might be coarser than the canonical topology when $V$ is not metrizable.

\begin{prop}
The topology of a complete metrizable $\mathbb M$-linear $\mathbb M$-module $M$ is identical to its canonical topology.
\end{prop}

\begin{proof}
It is sufficient to show that a null sequence $(s_n)_{n \in \mathbb N}$ in $M$ is also a null sequence for the canonical topology.
Let us consider the continuous $\mathbb M$-linear map
\[
\mathbb M^\mathbb N \to M, \quad (u_n)_{n \in \mathbb N} \mapsto \sum_{n=0}^\infty u_ns_n
\]
which is well defined since the topology of $M$ is $\mathbb M$-linear and $M$ is complete.
If $\mathbb M^\mathbb N$ is endowed with the product topology, then there exists a sequence of isomorphisms of topological $\mathbb M$-modules $\mathbb M^\mathbb N \simeq (\mathbb Z_\bullet^\mathbb N)^\mathbb N \simeq \mathbb Z_\bullet^{\mathbb N \times \mathbb N} \simeq \mathbb Z_\bullet^\mathbb N \simeq \mathbb M$.
It follows that the canonical topology on $\mathbb M^\mathbb N$ is identical to its product topology.
Since the obvious ``basis'' of $\mathbb M^\mathbb N$ is a null sequence, its image $(s_n)_{n \in \mathbb N}$ is also a null sequence for the canonical topology of $M$.
\end{proof}

\begin{cor} \label{metf}
\begin{enumerate}
\item \label{metf1}
Any $\mathbb M$-linear map $M \to N$ between $\mathbb M$-linear $\mathbb M$-modules with $M$ metrizable and $N$ complete is continuous,
\item \label{metf2} If $V$ is a complete \emph{metrizable} linear abelian group, then the topology of $V_\bullet$ inherited from $V$ is identical to its canonical topology as an $\mathbb M$-module. \qed
\end{enumerate}
\end{cor}

\begin{lem} \label{carb}
If $V$ is a complete linear abelian group, then
\[
V_\bullet \simeq \mathbb Z_\bullet \widehat \otimes_{\mathbb Z} V \simeq \mathrm{Hom}_{\mathbb Z\mhyphen\mathrm{cont}}(\mathbb Z^{\mathbb N}, V)
\]
(where $\mathbb Z_\bullet$ is endowed with the discrete topology and $\mathbb Z^{\mathbb N}$ with the product topology).
\end{lem}

\begin{proof}
The first isomorphism is $(s_n)_{n \in \mathbb N} \mapsto \sum_{n \in \mathbb N} e_n \otimes s_n$ whose inverse is given by $\sum_{n \in \mathbb N} a_n \otimes s_n \mapsto (a_ns_n)_{n \in \mathbb N}$.
Now, if we let $V'$ run through all open subgroups of $V$, we have from the first part
\[
V_\bullet \simeq \varprojlim (\mathbb Z_\bullet \otimes_{\mathbb Z} V)/(\mathbb Z_\bullet \otimes V')= \varprojlim (\mathbb Z_\bullet \otimes_{\mathbb Z} V/V') = \varprojlim (V/V')_\bullet
\]
 but also
\[
\mathrm{Hom}_{\mathbb Z\mhyphen\mathrm{cont}}(\mathbb Z^{\mathbb N}, V) = \varprojlim \mathrm{Hom}_{\mathbb Z\mhyphen\mathrm{cont}}(\mathbb Z^{\mathbb N}, V/V').
\]
We are therefore reduced to the case $V$ is discrete but then
\[
\mathrm{Hom}_{\mathbb Z\mhyphen\mathrm{cont}}(\mathbb Z^{\mathbb N}, V) \simeq V \cdot \mathbb N \simeq V_\bullet. \qedhere
\]
\end{proof}

\begin{rmk}
The second isomorphism of lemma \ref{carb} implies that
\[
\mathrm{Hom}_{\mathbb Z-\mathrm{cont}}(\mathbb Z^\mathbb N, V) \simeq \mathrm{Hom}_{\mathbb Z-\mathrm{cont}}(P, V) 
\]
if $P$ denotes as usual the \emph{Graev free abelian group of countable rank}.
This is somehow related to theorem \ref{intPr} below.
\end{rmk}

\begin{prop} \label{adcont}
There exists an adjunction (with respect to the canonical topology on $M(*)$)
\[
\mathrm{Hom}_{\mathbb Z\mhyphen\mathrm{cont}}(\widehat {M(*)}, V) \simeq \mathrm{Hom}_{\mathbb M}(M, V_\bullet)
\]
between $\mathbb M$-modules and complete linear abelian groups.
The functor $V \to V_\bullet$ preserves all direct sums of complete linear abelian groups.
\end{prop}

\begin{proof}
In order to show that the first assertion holds, it is sufficient to prove that there exists a natural isomorphism
\[
\mathrm{Hom}_{\mathbb Z\mhyphen\mathrm{cont}}(M(*), V) \simeq \mathrm{Hom}_{\mathbb M}(M, V_\bullet)
\]
(before completion) since $V$ itself is assumed to be complete.
Since both sides preserve all limits, we are reduced to the trivial case $M = \mathbb M$ that reads
\[
\mathrm{Hom}_{\mathbb Z\mhyphen\mathrm{cont}}(\mathbb Z^\mathbb N, V) \simeq V_\bullet.
\]
This was shown in lemma \ref{carb}.
Now, the same lemma tells us that that $V_{\bullet} = \mathbb Z_\bullet \widehat \otimes_{\mathbb Z} V$.
The second assertion therefore follows from the fact that both operations of tensorizing with $\mathbb Z_\bullet$ and completing preserve direct sums.
\end{proof}

As a consequence of the proposition, the functor $V \mapsto V_\bullet$ preserves all limits and all direct sums of complete linear abelian groups and the functor $M \mapsto \widehat M(*)$ preserves all colimits\footnote{Be careful that an epimorphism in the category of complete linear abelian groups need not be surjective.}.
Beware that, if $V$ is a complete linear abelian group, then the continuous bijective linear map $V_\bullet(*) \simeq V$ may not be a homeomorphism when the left hand side carries the canonical topology.

\begin{xmp}
\begin{enumerate}
\item
If $M$ is a finitely presented $\mathbb M$-module, then $M$ is complete for the canonical topology and therefore $\widehat{M(*)} = M(*)$.
\item
If $M := \mathbb M/\mathbb M^\vee$ (which is actually a ring), then $\widehat M = 0$ and therefore $\widehat{M(*)} = 0$.
\end{enumerate}
\end{xmp}


The next result is the analog in our situation of theorem 3.2 in \cite{Ren25}.

\begin{prop} \label{metff}
When restricted to complete \emph{metrizable} linear abelian groups, the functor $V \mapsto V_\bullet$ is fully faithful and preserves quotients.
\end{prop}

\begin{proof}
Since completion preserves quotients of metrizable linear abelian groups and $V_{\bullet} \simeq \mathbb Z_\bullet \widehat \otimes_{\mathbb Z} V$, only fullfaithfulness requires a proof.
It is sufficient to show that, if $V$ is a complete metrizable linear abelian group, then the continuous bijective linear map $V_\bullet(*) \simeq V$ is actually a homeomorphism (so that the former is complete).
But this follows from assertion \eqref{metf2} of corollary \ref{metf}.
\end{proof}

As a consequence of the propostion, the functor $V \mapsto V_\bullet$ sends \emph{strict}\footnote{A morphism is said to be \emph{strict} when regular coimage equals regular image: quotient topology identical to induced topology on the image for a continuous map.} exact sequences of complete metrizable linear groups to exact sequences of $\mathbb M$-modules.
In particular, it preserves \emph{short} exact sequences\footnote{A sequence $0 \to V' \overset i \to V \overset p \to V'' \to 0$ is a short exact sequence when $i = \ker p$ and $\mathrm{coker}\; i = p$.}.
Be careful however that this functor is not right exact since the inclusion map $\mathbb Z \rightarrowtail \mathbb Z_p$ (discrete topology on $\mathbb Z$ and $p$-adic topology on $\mathbb Z_p$) is an epimorphism in the category of complete metrizable linear groups but the map $\mathbb Z_\bullet \to \mathbb Z_{p\bullet}$ is not surjective.

\begin{prop} \label{laxmon}
The functor $V \mapsto V_\bullet$ is (lax) monoidal on complete linear abelian groups.
\end{prop}

\begin{proof}
We need to describe a natural map
\[
V_\bullet \otimes_{\mathbb Z_\bullet} W_\bullet \to (V \widehat \otimes_{\mathbb Z} W)_\bullet
\]
or, by adjunction,
\[
W_\bullet \to \mathcal H\mathrm{om}_{\mathbb Z_\bullet} (V_\bullet, (V \widehat \otimes_{\mathbb Z} W)_\bullet).
\]
After writing $W$ as a limit of discrete abelian groups with surjective transition maps, we see that we can assume $W$ discrete.
But there always exists a canonical map
\[
\begin{tikzcd}
W \ar[r] & \mathrm{Hom}_{\mathbb Z\mhyphen\mathrm{cont}}(V, V \widehat \otimes_{\mathbb Z} W) \ar[r] & \mathrm{Hom}_{\mathbb M} (V_\bullet, (V \widehat \otimes_{\mathbb Z} W)_\bullet) \ar[d, equal] \\ && \mathcal H\mathrm{om}_{\mathbb Z_\bullet} (V_\bullet, (V \widehat \otimes_{\mathbb Z} W)_\bullet)(*)
\end{tikzcd}
\]
and we can use the adjunction of proposition \ref{adj} since $W$ is now discrete.
\end{proof}

\begin{xmp}
The functor $V \mapsto V_\bullet$ is \emph{not} strong monoidal:
\[
\mathbb Q_\bullet \otimes_{\mathbb Z_\bullet} \mathbb Z[[t]]_\bullet = \mathbb Q \otimes_{\mathbb Z} \mathbb Z[[t]]_\bullet \neq \mathbb Q[[t]]_\bullet = (\mathbb Q \widehat \otimes_{\mathbb Z} \mathbb Z[[t]])_\bullet.
\]
The same example implies that the adjoint functor is not monoidal although $\mathbb Z^{\mathbb N} \widehat \otimes_{\mathbb Z} \mathbb Z^{\mathbb N} \simeq \mathbb Z^{\mathbb N \times \mathbb N}$.
\end{xmp}

\begin{prop}
If $V$ is a complete metrizable linear abelian group, then there exists a natural isomorphism (of abelian groups)
\[
(V_\bullet)^\vee \simeq \mathrm{Hom}_{\mathbb Z\mhyphen\mathrm{cont}} (V, \mathbb Z_\bullet).
\]
\end{prop}

\begin{proof}
We have
\begin{align*}
(V_\bullet)^\vee & = \mathcal H\mathrm{om}_{\mathbb Z_\bullet} (V_\bullet, \mathbb Z_\bullet)
\\ &\simeq \mathrm{Hom}_{\mathbb M}(\mathbb M, \mathcal H\mathrm{om}_{\mathbb Z_\bullet} (V_\bullet, \mathbb Z_\bullet))
\\ &\simeq \mathrm{Hom}_{\mathbb M}(V_\bullet, \mathcal H\mathrm{om}_{\mathbb Z_\bullet} (\mathbb M, \mathbb Z_\bullet))
\\ &\simeq \mathrm{Hom}_{\mathbb M}(V_\bullet, \mathbb M^\vee)
\\ &\simeq \mathrm{Hom}_{\mathbb M}(V_\bullet, (\mathbb Z \cdot \mathbb N)_\bullet)
\\ & \simeq \mathrm{Hom}_{\mathbb Z\mhyphen\mathrm{cont}} (V, \mathbb Z \cdot \mathbb N)
\\ & \simeq \mathrm{Hom}_{\mathbb Z\mhyphen\mathrm{cont}} (V, \mathbb Z_\bullet). \qedhere
\end{align*}
\end{proof}

During the proof of theorem \ref{cohM}, we used the discrete case of the following:

\begin{cor} \label{torlm}
Let $V$ be a complete metrizable linear abelian group.
If $V^\vee = 0$, then $(V_\bullet)^\vee = 0$. \qed
\end{cor}

\begin{proof}
$
(V_\bullet)^\vee \simeq \mathrm{Hom}_{\mathbb Z\mhyphen\mathrm{cont}} (V, \mathbb Z_\bullet) \rightarrowtail \mathrm{Hom}_{\mathbb Z\mhyphen\mathrm{cont}} (V, \mathbb Z^\mathbb N) \simeq (V^\vee)^\mathbb N = 0.
$
\end{proof}


\section{$\mathbb M$-rings}


We endowed in section \ref{dualsec} the category of $\mathbb M$-modules with an additive closed symmetric monoidal structure.
We send the reader to section 1.2 and 1.3 of \cite{Marty09} for a very nice presentation of the theory of (commutative) monoids in a (symmetric) monoidal category. 

\begin{dfn}
An (internal) \emph{$\mathbb M$-ring} $A$ is a monoid in the monoidal category of $\mathbb M$-modules and a (left) \emph{$A_{/\mathbb Z}$-module}\footnote{This notation is motivated by forthcoming development.} $M$ is then an $A$-object.
\end{dfn}

In other words, an $\mathbb M$-ring $A$ (resp.\ an $A_{/\mathbb Z}$-module $M$) is an $\mathbb M$-module endowed with an associative unitary\footnote{Both sides in the first case.} morphism $A \otimes_{\mathbb Z_\bullet} A \to A$ (resp.\ $A \otimes_{\mathbb Z_\bullet} M \to M$).
The category of $\mathbb M$-rings has all limits and colimits and the forgetful functor to $\mathbb M$-modules admits an adjoint.
The category of $A_{/\mathbb Z}$-modules is abelian with all limits and colimits.
If $A \to B$ is a morphism of $\mathbb M$-rings, then the forgetful functor has both an adjoint and a coadjoint.
If $A$ is an $\mathbb M$-ring, then there exists an opposite $M$-ring $A^{\mathrm{op}}$ obtained by switching factors and a \emph{right $A_{/\mathbb Z}$-module} is the same thing as a left $A^{\mathrm{op}}_{/\mathbb Z}$-module.
The tensor product of a right $A_{/\mathbb Z}$-module $M$ with a left $A_{/\mathbb Z}$-module $N$ is provided by the right exact sequence
\[
M \otimes_{\mathbb Z_\bullet} A \otimes_{\mathbb Z_\bullet} N \rightrightarrows M \otimes_{\mathbb Z_\bullet} N \to M \otimes_A N.
\]
The ring $A$ is \emph{commutative} when $A^{\mathrm{op}} = A$ in which case the category of $A$-modules is closed symmetric monoidal.
There is \emph{no} multiplication map $A \times A \to A$ or $A \times M \to M$ giving rise to the $\mathbb M$-ring structure or to the $A_{/\mathbb Z}$-module structure in general but the unit does come from some $1 \in A(*)$.
Also, it often happens in practice that the category of $A_{/\mathbb Z}$-modules is (isomorphic to) a category of genuine modules (over some usually huge non-commutative ring).

Recall from definition \ref{linab} that a \emph{linear ring} $R$ (resp.\ a \emph{linear $R$-module} $V$) is a topological ring (resp.\ a topological $R$-module) whose topology is linear: the open subgroups form a basis of neighborhoods of zero.
For example, an adic ring (resp.\ a Huber ring) $R$ is a metrizable linear ring and an $R$-module (resp a finitely presented\footnote{There is no natural topology otherwise.} $R$-module) $V$ is a metrizable linear $R$-module.
Note that we only require a basis of neighborhoods of zero made of \emph{subgroups} and not of ideals or submodules in general.

\begin{prop} \label{fstMrg}
\begin{enumerate}
\item
There exists a functor $R \mapsto R_\bullet$ (resp.\ $V \mapsto V_\bullet$) from complete linear rings (resp.\ complete linear $R$-modules) to $\mathbb M$-rings (resp.\ $R_{\bullet/\mathbb Z}$-modules) that preserves all limits.
\item Restricted to \emph{metrizable} rings (resp.\ modules), it is fully faithful and preserves quotients.
\item Restricted to \emph{discrete} rings (resp.\ modules), it has a coadjoint $A \mapsto A(*)$ (resp.\ $M \mapsto M(*)$) that preserves all limits and all colimits.
\item If $A$ is an $\mathbb M$-ring, then there exists a forgetful functor from $A_{/\mathbb Z}$-modules to genuine $\mathbb M \otimes_{\mathbb Z} A(*)$-modules.
\item If $R$ is a discrete ring, then we actually get an isomorphism of categories between $R_{\bullet/\mathbb Z}$-modules and genuine $\mathbb M \otimes_{\mathbb Z} R$-modules\footnote{aka $\mathbb M$-$R^{\mathrm{op}}$-bimodules.}.
\end{enumerate}
\end{prop}

\begin{proof}
The first two assertions follow from propositions \ref{laxmon}, \ref{adcont} and \ref{metff} and the fact that a complete linear ring is a monoid in the monoidal category of complete linear abelian groups.
The other assertions result from proposition \ref{monadj}.
For the last ones, on uses the fact that a morphism of $\mathbb M$-rings $A \to \mathcal E\mathrm{nd}_{\mathbb Z_\bullet}(M)$ induces a morphism of rings $A(*) \to \mathcal E\mathrm{nd}_{\mathbb Z_\bullet}(M)(*) = \mathrm{End}_{\mathbb M}(M)$ and conversely when $A = R_\bullet$ with $R$ discrete.
\end{proof}

When $R$ is a discrete ring, do not confuse the ring $\mathbb M \otimes_{\mathbb Z} R$ with the ring $\mathbb M_R$ of column-finite matrices with coefficients in $R$ which is much bigger (but see below).

\begin{xmps}
\begin{enumerate}
\item $\mathbb Z_\bullet$ is a commutative $\mathbb M$-ring.
A $\mathbb Z_{\bullet/\mathbb Z}$-module is the same thing as an $\mathbb M$-module.
Be careful not to confuse external with internal Hom:
\[
\mathrm{Hom}_{\mathbb Z_{\bullet}}(M, N) = \mathrm{Hom}_{\mathbb M}(M, N) = \mathcal H\mathrm{om}_{\mathbb Z_\bullet}(M, N)(*) \neq \mathcal H\mathrm{om}_{\mathbb Z_\bullet}(M, N).
\]
\item Even if $\mathbb M^\vee$ is only a two-sided ideal of $\mathbb M$, there exists a canonical isomorphism $\mathbb M^\vee \simeq \mathbb Z[X]_\bullet$ with a commutative $\mathbb M$-ring (sending $t$ to $X$).
A $\mathbb Z[X]_{\bullet/\mathbb Z}$-module is the same thing as a genuine $\mathbb M[X]$-module (equivalently an $\mathbb M$-module endowed with an $\mathbb M$-linear map).
\item There exists a similar isomorphism $\mathbb M \simeq \mathbb Z[[X]]_\bullet$ (where $\mathbb Z[[X]]$ is endowed with the $X$-adic topology) so that $\mathbb M$ has a natural structure of commutative $\mathbb M$-ring besides its original structure of non-commutative (usual) ring.
Be careful however not to confuse an $\mathbb M$-module with a $\mathbb Z[[X]]_{\bullet/\mathbb Z}$-module which is the same thing as a genuine $\mathbb M[[X]]$-module (see below).
\item It follows from corollary \ref{tensfr} that
\[
\mathbb Z[[X]]_\bullet \otimes^{\mathrm L}_{\mathbb Z_\bullet} \mathbb Z[[Y]]_\bullet \simeq \mathbb Z[[X, Y]]_\bullet.
\]
As a consequence,
\[
\mathbb Z_{p\bullet} \otimes^{\mathrm L}_{\mathbb Z_\bullet} \mathbb Z[[X]]_\bullet \simeq \mathbb Z_p[[X]]_\bullet
\]
(for the $(p,X)$-adic topology) but also
\[
\mathbb Z_{p\bullet} \otimes^{\mathrm L}_{\mathbb Z_\bullet} \mathbb Z_{p\bullet} \simeq \mathbb Z_{p\bullet} \quad \mathrm{and} \quad \mathbb Z_{p\bullet} \otimes^{\mathrm L}_{\mathbb Z_\bullet} \mathbb Z_{\ell\bullet} = 0
\]
when $\ell \neq p$.
Since
$\mathbb Q_{p\bullet} = \varinjlim_p \mathbb Z_{p\bullet}$, we also have
\[
\mathbb Q_{p\bullet} \otimes^{\mathrm L}_{\mathbb Z_\bullet} \mathbb Z[[X]]_\bullet \simeq \mathbb Q_p[[X]]^{\mathrm{bd}}_\bullet \quad \mathrm{and} \quad \mathbb Q_{p\bullet} \otimes^{\mathrm L}_{\mathbb Z_\bullet} \mathbb Q_{p\bullet} \simeq \mathbb Q_{p\bullet}
\]
where $\mathbb Q_p[[X]]^{\mathrm{bd}} := \mathbb Z_p[[X]][1/p]$ denotes the ring of bounded functions on the open unit disc.
\item The category of $\mathbb Z_{p\bullet/\mathbb Z}$-modules (resp.\ $\mathbb Q_{\bullet/\mathbb Z}$-modules, resp.\ $\mathbb Q_{p\bullet/\mathbb Z}$-modules) is a full subcategory of the whole category of $\mathbb M$-modules.
In other words, being a $\mathbb Z_{p\bullet/\mathbb Z}$-modules (resp.\ $\mathbb Q_{\bullet/\mathbb Z}$-modules, resp.\ $\mathbb Q_{p\bullet/\mathbb Z}$-modules) is a property of the $\mathbb M$-module and not an extra structure.
This follows from the fact that, in each case, the $\mathbb M$-ring is idempotent\footnote{An object $M$ in a symmetric monoidal category is said to be \emph{idempotent} is $M \otimes M \simeq M$.}.
\end{enumerate}
\end{xmps}

In order to go further, we shall need the following technical result:

\begin{lem} \label{limfn}
If $R$ is an $I$-adic ring with $I = (f_1, \ldots, f_d)$ and we equip the free $R$-module $R \cdot \mathbb N := \oplus_{\mathbb N} R$ with $I$-adic topology, then\footnote{We use multi index notation $\underline f^{\underline k} = \prod_{i=1}^df_i^{k_i}$ and $|\underline k| = \sum_{i=0}^n k_i$.}
\[
\widehat{R \cdot \mathbb N}_\bullet \simeq \varinjlim_{|\underline {k_n}| \to \infty} \prod_n \underline f^{\underline k_n} \widehat R_\bullet.
\]
\end{lem}

\begin{proof}
Since we always have
\[
I^k = \sum_{|\underline k| = k} \underline f^{\underline k} R,
\]
it is sufficient to show that\footnote{It is not necessary to assume $I$ of finitely generated here.}
\[
\widehat{R \cdot \mathbb N}_\bullet \simeq \varinjlim_{k_n \to \infty} \prod_n I^{k_n} \widehat R_\bullet.
\]
Now, we already know that the conclusion holds without the bullets:
\[
\widehat{R \cdot \mathbb N} \simeq \varinjlim_{k_n \to \infty} \prod_n I^{k_n} \widehat R.
\]
There exists therefore an injective map
\[
\varinjlim_{k_n \to \infty} \prod_n I^{k_n} \widehat R_{\bullet} \hookrightarrow \widehat{R \cdot \mathbb N}_\bullet
\]
and it remains to show that it is surjective.
If $(\varphi_m)_{m \in \mathbb N} \in \widehat{R \cdot \mathbb N}_\bullet$, then this is a null sequence and there exists therefore $l_m \to \infty$ such that $\varphi_m \in I^{l_m} \prod_n \widehat R$ for all $m \in \mathbb N$.
On the other hand, for all $m \in \mathbb N$, there exits $k_{m,n} \to \infty$ such that $\varphi_m \in \prod_n I^{k_{m,n}} \widehat R$.
We can then set $k_n = \inf_m\max \{k_{m,n}, l_m\}$.
\end{proof}

A complete adic ring $R$ is said to be \emph{formally finitely generated} if there exists a finitely generated ideal of definition $I$ such that $R/I$ is a finitely generated ring.
Equivalently, $R$ is a quotient of $\mathbb Z[T_1, \ldots, T_m][[X_1, \dots, X_d]]$ equipped with the $(X_1, \dots, X_d)$-adic topology.
Equivalently again, $R$ is an adic completion of a finitely generated ring.

We can then improve on the last assertion of proposition \ref{fstMrg} (which is exactly the discrete case when $R$ is finitely generated):

\begin{prop} \label{fft}
If an adic ring $R$ is formally finitely generated, then the category of $R_{\bullet/\mathbb Z}$-modules is isomorphic to the category of genuine\footnote{Completion is meant with respect to the topology of $R$.} $\mathbb M \widehat \otimes_{\mathbb Z} R$-modules.
\end{prop}

\begin{proof}
As a first step, we shall prove that there exists an isomorphism of right $\mathbb M$-modules
\[
\mathbb M^\mathbb N \otimes_{\mathbb M} R_\bullet \simeq \mathbb M \widehat \otimes_{\mathbb Z} R.
\]
We may assume that $R = \mathbb Z[T_1, \ldots, T_m][[X_1, \dots, X_d]]$ and we shall simply write $\mathbb Z[T][[X]]$ to lighten the notations (and avoid underlining everything).
Recall that
\[
\mathbb Z[T][[X]] \simeq \widehat{\mathbb Z[[X]][T]} 
\]
(where completion is meant with respect to $X$-adic completion).
Thanks to lemma \ref{limfn} and the fact that each $\prod_n T^{k_n} X^n\mathbb Z[[X]]_\bullet$ is a finitely presented $\mathbb M$-module,
\begin{align*}
\mathbb M^\mathbb N \otimes_{\mathbb M} \mathbb Z[ T][[ X]]_\bullet & \simeq \mathbb M^\mathbb N \otimes_{\mathbb M} \varinjlim_{k_n \to \infty} \prod_n T^{k_n} X^n\mathbb Z[[X]]_\bullet
\\ & \simeq \varinjlim_{k_n \to \infty} \mathbb M^\mathbb N \otimes_{\mathbb M} \prod_n T^{k_n} X^n\mathbb Z[[X]]_\bullet
\\ & \simeq \varinjlim_{k_n \to \infty} \left(\prod_n T^{k_n} X^n\mathbb Z[[X]]_\bullet\right)^\mathbb N
\\ & \simeq \varinjlim_{k_n \to \infty} \prod_n T^{k_n} X^n\mathbb M[[X]]
\\ & \simeq \mathbb M[T][[X]]
\\ & \simeq \mathbb M \widehat \otimes_{\mathbb Z} \mathbb Z[ T][[ X]].
\end{align*}
The identification $\mathbb Z[[X]]_\bullet^\mathbb N \simeq \mathbb Z_\bullet^\mathbb N[[X]] \simeq \mathbb M[[X]]$ follows from the fact that the functor $V \mapsto V_\bullet$ preserves all limits of topological abelian groups - and in particular arbitrary products.

Now, by definition, an $R_{\bullet/\mathbb Z}$-module (resp.\ an $\mathbb M \widehat \otimes_{\mathbb Z} R$-module) is an $\mathbb M$-module endowed with a unitary associative morphism $R_\bullet \otimes_{\mathbb Z_\bullet} M \to M$ (resp.\ $\mathbb M \widehat \otimes_{\mathbb Z} R \otimes_{\mathbb M} M \to M$).
And a morphism is in both cases an $\mathbb M$-linear map compatible with these actions.
Since
\[
\mathbb M^{\otimes 2} \otimes_{\mathbb M} R_\bullet \simeq \mathbb M \widehat \otimes_{\mathbb Z} R,
\]
we have:
\[
R_\bullet \otimes_{\mathbb Z_\bullet} M = (\mathbb M^{\otimes 2} \otimes_{\mathbb M} R_\bullet) \otimes_{\mathbb M} M \simeq \mathbb M \widehat \otimes_{\mathbb Z} R \otimes_{\mathbb M} M. \qedhere
\]
\end{proof}

In order to state the next corollary, it is necessary to recall some notions and introduce some notations.

If $R$ is a complete \emph{adic} ring, we shall then denote by $\mathbb M_R$ the set of \emph{column-null} matrices with coefficients in $R$, meaning that all columns are made of null sequences of elements of $R$.
This is a ring for usual multiplication and $\mathbb M_R \simeq \mathrm{End}_{R\mhyphen\mathrm{cont}}(R_\bullet)$.
In particular, $R_\bullet$ is an $\mathbb M_R$-module and we have $\mathbb M_R \simeq R_\bullet^\mathbb N$.
In the case $R$ is discrete, this is the same thing as the ring of \emph{column-finite} matrices with coefficients in $R$ so that, in particular, $\mathbb M_{\mathbb Z} = \mathbb M$.
If we complete with respect to some ideal $I$ of $R$, then $\mathbb M_{\widehat R} = \widehat{\mathbb M_R}$ (since null sequences preserve all limits).

An adic ring is said to be \emph{strictly formally finitely generated} if it is (isomorphic to) a quotient of $\mathbb Z[[X_1, \dots, X_d]]$ equipped with the $(X_1, \dots, X_d)$-adic topology.
As a standard example, we may of course consider $\mathbb Z$ or $\mathbb Z[[X]]$ but also $\mathbb Z_p$ or $\mathbb Z_p[[X]]$ (but not $\mathbb Z[X]$).

If $R$ is a (complete) linear ring, then the \emph{ring of convergent power series} over $R$ is the completion $R\{X_1, \ldots, X_d\}$ of the polynomial ring $R[X_1, \ldots, X_d]$ with respect to the subgroups $V[X_1, \ldots, X_d]$ when $V$ runs through the open subgroups of $R$.
In other words,
\[
R\{X_1, \ldots, X_d\} = \left\{\sum_{\mathbb n \in \mathbb N^d}\underline f_{\underline n} \underline X^{\underline n}, \quad f_{\underline n} \to 0\right \} \subset R[[X_1, \ldots, X_d]].
\]

\begin{cor} \label{fft}
If an adic ring $R$ is \emph{strictly} formally finitely generated, then the category of $R\{X_1, \ldots, X_d\}_{\bullet/\mathbb Z}$-modules is equivalent to the category of genuine $\mathbb M_R\{X_1, \ldots, X_d\}$-modules.
\end{cor}

\begin{proof}
From the identification $\mathbb Z_\bullet^\mathbb N[[X]] \simeq \mathbb Z[[X]]_\bullet^\mathbb N$, we deduce that
\[
\mathbb M \widehat \otimes_{\mathbb Z} \mathbb Z[[X]] \simeq \mathbb M[[X]] \simeq \mathbb M_{\mathbb Z[[X]]}.
\]
It follows that $\mathbb M \widehat \otimes_{\mathbb Z} R\{X_1, \ldots, X_d\} \simeq \mathbb M_R\{X_1, \ldots, X_d\}$.
\end{proof}

This applies to $\mathbb Z[[X]]$ or $\mathbb Z_p\{X\}$ for example but we can also easily deduce that the category of $\mathbb Q[[X]]^{\mathrm{bd}}_{\bullet/\mathbb Z}$-modules (resp.\ $\mathbb Q_{p\bullet/\mathbb Z}$-modules) is equivalent to the category of $\mathbb Q \otimes_{\mathbb Z} \mathbb M[[X]]$-modules (resp.\ $\mathbb M_p[1/p]$-modules).

Moving to $\mathbb M$-modules replaces completed tensor product of Banach spaces with the simpler tensor product we may think of:

\begin{prop}
If $E$ and $F$ are two $\mathbb Q_p$-Banach spaces, then\footnote{The derived decoration plays no role since everything is flat here.}
\[
E_\bullet \otimes^L_{\mathbb Z_\bullet} F_\bullet \simeq (E \widehat \otimes_{\mathbb Q_p} F)_\bullet.
\]
\end{prop}

\begin{proof}
We can write $E = \mathbb Q \otimes_{\mathbb Z} \widehat{\mathbb Z \cdot I}$ and $F = \mathbb Q \otimes_{\mathbb Z} \widehat{\mathbb Z \cdot J}$ (where completion is meant with respect to $p$-adic topology) and it is sufficient to prove that
\begin{equation} \label{Zlevl}
\widehat{\mathbb Z \cdot I}_\bullet \otimes^L_{\mathbb Z_\bullet} \widehat{\mathbb Z \cdot I}_\bullet \simeq \widehat{\mathbb Z \cdot (I \times J)}_\bullet.
\end{equation}
Moreover, since
\[
\widehat{\mathbb Z \cdot I} = \varinjlim \widehat{\mathbb Z \cdot I'}
\]
when $I'$ runs through countable subsets, we may assume that $I = \mathbb N$ and, by symmetry, that $J = \mathbb N$ (the case $I$ or $J$ finite is trivial).
Thanks to lemma \ref{limfn}, we are therefore reduced to proving
\[
\prod_n p^{k_n} \mathbb Z_{p\bullet} \otimes^L_{\mathbb Z_\bullet} \prod_n p^{l_n} \mathbb Z_{p\bullet} = \prod_n p^{k_n} \mathbb Z_{p\bullet} \times \prod_n p^{l_n} \mathbb Z_{p\bullet}
\]
which follows from corollary \ref{tensfr}.
\end{proof}

\begin{xmp}
We have 
\[
\mathbb Q_p\{X\}_\bullet \otimes_{\mathbb Z_\bullet} \mathbb Q_p\{Y\}_\bullet \simeq \mathbb Q_p\{X, Y\}_\bullet
\]
but also
\[
\mathbb Z_p\{X\}_\bullet \otimes_{\mathbb Z_\bullet} \mathbb Z_p\{Y\}_\bullet \simeq \mathbb Z_p\{X, Y\}_\bullet,
\]
using the intermediate step \eqref{Zlevl}.
\end{xmp}

\begin{prop} \label{idmod}
$Z[[X]]_\bullet$ is a (derived) idempotent $\mathbb Z[X]_{\bullet/\mathbb Z}$-module:
\[
\mathbb Z[[X]]_\bullet \otimes^{\mathrm L}_{\mathbb Z[X]_\bullet } \mathbb Z[[X]]_\bullet \simeq \mathbb Z[[X]]_\bullet.
\]
\end{prop}

\begin{proof}
By definition of relative tensor product in a closed monoidal category, if we let $\mathbb Z[X]_\bullet$ act on $\mathbb Z[[X]]_\bullet \otimes_{\mathbb Z_\bullet } \mathbb Z[[X]]_\bullet$ by $X \mapsto 1 \otimes X - X \otimes 1$, then we have
\begin{align*}
\mathbb Z[[X]]_\bullet \otimes^{\mathrm L}_{\mathbb Z[X]_\bullet } \mathbb Z[[X]]_\bullet &\simeq \left(\mathbb Z[[X]]_\bullet \otimes^{\mathrm L}_{\mathbb Z_\bullet } \mathbb Z[[X]]_\bullet \right) \otimes^{\mathrm L}_{\mathbb Z[X]_\bullet } \mathbb Z_\bullet
\\ &\simeq \mathbb Z[[X, Y]]_\bullet \otimes^{\mathrm L}_{\mathbb Z[X]_\bullet } \mathbb Z_\bullet
\\ &\simeq \mathbb Z[[X]]_\bullet
\end{align*}
using the resolution $\mathbb Z[X]_\bullet \overset X \to \mathbb Z[X]_\bullet$ of $\mathbb Z_\bullet$.
\end{proof}

As a consequence, the forgetful functor from $\mathbb Z[[X]]_{\bullet/\mathbb Z}$-modules to $\mathbb Z[X]_{\bullet/\mathbb Z}$-modules is fully faithful.
We may therefore consider a $\mathbb Z[[X]]_{\bullet/\mathbb Z}$-module as an $\mathbb M$-module endowed with an operator satisfying an extra-property (but \emph{no} extra structure).

\begin{cor}
Let $R$ be a finitely generated ring and $I,J$ two ideals of $R$, then (if the exponents indicates that we are completing with respect to the given ideal)
\[
\widehat {R}^{/I}_\bullet \otimes_{R_\bullet} \widehat R^{/J}_\bullet \simeq \widehat R^{/I+J}_\bullet. \qed
\]
\end{cor} 

We finish with the analog of a theorem of Mann.
Recall that if $R$ is a discrete ring, then an $R$-module (or complex) $M$ is \emph{derived complete} with respect to a finitely generated ideal $I$ if $\mathrm R\varprojlim_f M = 0$ when $f \in I$.
When $M$ is a Hausdorff $I$-adic $R$-module, this is equivalent to $M$ being complete.

\begin{thm}
If $R$ is a \emph{finitely generated} ring, $I \subset R$ an ideal and $M, N$ are two derived $I$-complete $R_{\bullet/\mathbb Z}$-modules, then $M \otimes^{\mathrm L}_{R_\bullet} N$ also is derived $I$-complete.
\end{thm}

\begin{proof}
This is obtained from proposition \ref{idmod} with the same strategy as Mann in lemma 2.12.19 of \cite{Mann22} (see also theorem 9.3.5 in \cite{Kedlaya25}).
\end{proof}




\section{$\mathbb M$-modules and the condensed world}

A \emph{light\footnote{\emph{Heavy} if you replace metrizable with Hausdorff.} condensed set} is a sheaf (of sets) on the site of compact metrizable\footnote{Or equivalently second-countable if already heavy.} spaces covered by jointly surjective finite families.
Equivalently, one can use only metrizable Stone spaces (aka light profinite sets aka totally disconnected compact metrizable spaces).
However, taking into account all compact metrizable spaces from the beginning sounds more natural in the sense that we recover \emph{the pretopos} that generates the topos (which is therefore \emph{coherent}).
There exists an obvious functor sending a topological space $X$ to a light condensed set $\underline X$ which is defined by
\[
\underline X(S) = \mathrm{Hom}_{\mathrm{cont}}(S, X)
\]
for $S$ compact metrizable.
This functor has an adjoint $\mathcal X \mapsto \mathcal X(*)$ -- where $*$ denotes the one point space -- and becomes fully faithful when restricted to sequential topological spaces.
The adjunction may be enriched to provide an isomorphism
\[
\underline{\mathrm{Hom}_{\mathrm{cont}}(\mathcal X(*), Y)} \simeq \mathcal H \mathrm{om}(\mathcal X, \underline Y)
\]
with respect to the compact-open topology (when $Y$ is equiped with its $k$-topology).

A \emph{light condensed abelian group} is an abelian group in the cartesian category of light condensed sets (or equivalently a sheaf of abelian groups on the site of compact metrizable spaces and jointly surjective finite families).
If $M$ is a topological abelian group, then $\underline M$ is naturally a light condensed abelian group and the above isomorphism then provides
\[
\underline{\mathrm{Hom}_{\mathbb Z\mhyphen\mathrm{cont}}(M, N)} \simeq \mathcal H \mathrm{om}_{\underline{\mathbb Z}}(\underline M, \underline N)
\]
when $N$ is any topological abelian group and $M$ is sequential (and we are using the compact-open topology with respect to the $k$-topology on $N$).
Be careful that the topology of $\mathcal M(*)$ need \emph{not} be compatible with its structure as an abelian group when $\mathcal M$ is a light condensed abelian group.

We shall systematically use the following fundamental fact (recall that, when $X$ is a topological space, we denote Markov free abelian group as $\mathbb Z \cdot X$).

\begin{prop} \label{untop}
If $S$ is a compact metrizable space, then $\underline{\mathbb Z \cdot S}$ is universal for morphisms from $\underline S$ to light condensed abelian groups\footnote{In other words, $\underline{\mathbb Z \cdot S} = \mathbb Z[\underline S]$ in Clauzen-Scholze notations.}.
\end{prop}

\begin{proof}
This is done by hand. See for example example 5.1.7 of \cite{Kedlaya25} or proposition 2.1 and exercise 2.3 of \cite{Scholze26b}.
\end{proof}

The statement may be rephrased as follows: if $\mathcal M$ is a light condensed abelian group, then there exists a natural isomorphism
\[
\mathrm{Hom}_{\underline {\mathbb Z}}(\underline{\mathbb Z \cdot S}, \mathcal M) \simeq \mathcal M(S).
\]

Nöbeling's theorem \ref{Nobel} implies that the bidual $\mathbb Z\cdot S^{\vee\vee} := (\mathbb Z\cdot S)^{\vee\vee} $ is a complete linear abelian group when $S$ is compact Hausdorff and we may therefore consider $\mathbb Z\cdot S^{\vee\vee}_\bullet :=(\mathbb Z\cdot S^{\vee\vee})_\bullet$.

\begin{prop} \label{expfrm}
\begin{enumerate}
\item
If $M$ is an $\mathbb M$-module, then the presheaf defined on compact metrizable spaces as
\[
M^{\mathrm{cond}} : S \mapsto \mathrm{Hom}_{\mathbb M}(\mathbb Z\cdot S^{\vee\vee}_\bullet, M)
\]
is a light condensed abelian group.
\item The functor $M \mapsto M^{\mathrm{cond}}$ is fully faithful and preserves all limits as well as all colimits.
\end{enumerate}
\end{prop}

\begin{proof}
In order to prove the first assertion, it is sufficient to show that the functor $M^{\mathrm{cond}}$ is left exact (it turns finite colimits into limits).
First of all, the reflection $S \mapsto \pi_0(S)$ (on the category of all topological spaces) preserves all colimits and we may therefore assume that $S$ is a metrizable \emph{Stone} space.
We then know from proposition \ref{ddcomp} that $\mathbb Z\cdot S^{\vee\vee} \simeq \mathbb Z\cdot S^{(\wedge)}$ is the linear completion of $\mathbb Z\cdot S$.
Since completion preserves strict exact sequences of metrizable linear abelian groups, we only have to make sure that the functor $S \mapsto \mathbb Z \cdot S^{\mathrm{lin}}$ is right exact on metrizable Stone spaces (the superscript indicates that we use the linear topology).
Since this is already the case when we forget the topology, it is sufficient to notice that $\mathbb Z \cdot S \twoheadrightarrow \mathbb Z \cdot S'$ is a quotient map for the linear topology when $S' \twoheadrightarrow S$ is a continuous surjective map of Stone spaces.

Now, Nöbeling's theorem \ref{Nobel} states that there exists a (highly non-canonical) isomorphism $\mathbb Z\cdot S^{\vee\vee}_\bullet \simeq \mathbb Z_\bullet^I$ for some countable set $I$.
This implies that $M^{\mathrm{cond}}(S)$ is isomorphic either to $M$ or to $M(*)^n$ (naturally in $M$) for fixed $S$ (depending on $S$ being infinite or not).
In particular, the second assertion already holds at the presheaf level.
\end{proof}

By definition, there exists a natural isomorhism
\begin{equation} \label{adjpoor}
\mathrm{Hom}_{\mathbb M}(\mathbb Z\cdot S^{\vee\vee}_\bullet, M) \simeq \mathrm{Hom}_{\underline{\mathbb Z}}(\underline {\mathbb Z\cdot S}, M^{\mathrm{cond}}).
\end{equation}

\begin{prop} \label{fstcomp}
\begin{enumerate}
\item
If $M$ is an $\mathbb M$-module, then $M^{\mathrm{cond}}(*) \simeq M(*)$.
\item If $V$ is a complete metrizable linear abelian group, then $V_\bullet^{\mathrm{cond}} = \underline V$.
\end{enumerate}
\end{prop}

\begin{proof}
By definition,
\[
M^{\mathrm{cond}}(*) = \mathrm{Hom}_{\mathbb M}(\mathbb Z_\bullet, M) \simeq M(*).
\]

Now, if $S$ is a metrizable Stone space, then ($\mathbb Z\cdot S^{\vee\vee} \simeq \mathbb Z\cdot S^{(\wedge)}$ and)
\begin{align*}
V_\bullet^{\mathrm{cond}}(S) & = \mathrm{Hom}_{\mathbb M}( \mathbb Z\cdot S^{(\wedge)}_\bullet, V_\bullet)
\\ &\simeq \mathrm{Hom}_{\mathbb Z-\mathrm{cont}}(\mathbb Z\cdot S^{(\wedge)}, V)
\\ &\simeq \mathrm{Hom}_{\mathbb Z-\mathrm{cont}}(\mathbb Z\cdot S, V)
\\ &\simeq \mathrm{Hom}_{\mathrm{cont}}(S, V)
\\ &\simeq \underline V(S). \qedhere
\end{align*}
\end{proof}

As explained in the appendix, a functor from $\mathbb M$-modules to light condensed abelian groups that preserves all colimits must come from a unique right $\mathbb M$-module in the additive category of light condensed abelian groups (the image of $\mathbb M$ by the given functor).
Since, in our case, $\mathbb M \simeq \mathbb Z_\bullet^\mathbb N$, proposition \ref{fstcomp} tells us that this right module is $\underline {\mathbb Z^\mathbb N}$.
In particular, we see that the coadjoint to $M \mapsto M^{\mathrm{cond}}$ is explicitly given by $\mathcal M \mapsto \mathrm{Hom}_{\underline{\mathbb Z}}(\underline{\mathbb Z^\mathbb N}, \mathcal M)$ and that we can also write (with the notations of the appendix) $M^{\mathrm{cond}} = \underline{\mathbb Z^\mathbb N} \otimes_{\mathbb M}M$.
Since the functor $M \mapsto M^{\mathrm{cond}}$ also preserves all limits, it has a more mysterious adjoint\footnote{It may be called \emph{concrete solidification}.} that we shall denote as $\mathcal M \mapsto \mathcal M^{(\blacksquare)}$ so that
\[
\mathrm{Hom}_{\mathbb M}(\mathcal M^{(\blacksquare)}, N)\simeq \mathrm{Hom}_{\underline {\mathbb Z}}(\mathcal M, N^{\mathrm{cond}}).
\]
There actually exists an enriched adjunction:

\begin{prop}
If $\mathcal M$ is a light condensed abelian group and $N$ is an $\mathbb M$-module, then there exists a natural isomorphism
\[
\mathcal H \mathrm{om}_{\mathbb M}(\mathcal M^{(\blacksquare)}, N)^{\mathrm{cond}} \simeq \mathcal H\mathrm{om}_{\underline {\mathbb Z}}(\mathcal M, N^{\mathrm{cond}}).
\]
\end{prop}

\begin{proof}
Since light consensed abelian groups are abelian sheaves on the category of compact metrizable spaces, we may assume that $\mathcal M = \underline{\mathbb Z \cdot S}$ where $S$ is a compact metrizable space (they form a system of generators and our functors preserve all limits).
Moreover, the isomorphism \eqref{adjpoor} shows that
\begin{equation}
\underline{\mathbb Z \cdot \mathbb S}^{(\blacksquare)} = \mathbb Z\cdot S^{\vee\vee}_\bullet.
\end{equation}
We are therefore reduced to show:
\begin{equation}
\mathcal H \mathrm{om}_{\mathbb Z_\bullet}(\mathbb Z\cdot S^{\vee\vee}_\bullet, M)^{\mathrm{cond}} \simeq \mathcal H \mathrm{om}_{\underline {\mathbb Z}}(\underline {\mathbb Z \cdot S}, M^{\mathrm{cond}}).
\end{equation}
The point is that, although the functor $V \mapsto V_\bullet$ is only \emph{lax} monoidal, Nöbeling's theorem implies hat
\[
\mathbb Z\cdot T^{\vee\vee}_\bullet \otimes_{\mathbb Z_\bullet}\mathbb Z\cdot S^{\vee\vee}_\bullet \simeq \mathbb Z\cdot (S \times T)^{\vee\vee}_\bullet
\]
whenever $S$ and $T$ are compact Hausdorff spaces.
One can then compute when $T$ is a compact metrizable space:
\begin{align*}
\mathcal H \mathrm{om}_{\mathbb Z_\bullet}(\mathbb Z\cdot S^{\vee\vee}_\bullet, M)^{\mathrm{cond}}(T) & \simeq \mathrm{Hom}_{\mathbb Z_\bullet}(\mathbb Z\cdot T^{\vee\vee}_\bullet, \mathcal H \mathrm{om}_{\mathbb Z_\bullet}(\mathbb Z\cdot S^{\vee\vee}_\bullet, M)^{\mathrm{cond}})
\\ & \simeq \mathrm{Hom}_{\mathbb Z_\bullet}(\mathbb Z\cdot T^{\vee\vee}_\bullet \otimes_{\mathbb Z_\bullet}\mathbb Z\cdot S^{\vee\vee}_\bullet, M)^{\mathrm{cond}})
\\ & \simeq \mathrm{Hom}_{\mathbb Z_\bullet}(\mathbb Z\cdot (S \times T)^{\vee\vee}_\bullet, M)^{\mathrm{cond}}
\\ & \simeq M^{\mathrm{cond}}(S \times T)
\\ & \simeq \mathcal H \mathrm{om}_{\underline {\mathbb Z}}(\underline {\mathbb Z \cdot S}, M^{\mathrm{cond}})(T)
\end{align*}
(the last equality being completely standard in any topos - taking into account proposition \ref{untop}).
\end{proof}

As a consequence of the proposition, the functor $M \mapsto M^{\mathrm{cond}}$ is \emph{internally} fully faithful:
given two $\mathbb M$-modules $M$ and $N$, we have
\[
\mathcal H\mathrm{om}_{\mathbb Z_\bullet}(M, N)^{\mathrm{cond}} \simeq \mathcal H\mathrm{om}_{\underline {\mathbb Z}}(M^{\mathrm{cond}} , N^{\mathrm{cond}} ).
\]

\begin{cor} \label{frst}
The functor $M \mapsto M^{\mathrm{cond}}$ is (lax) monoidal.
\end{cor}

\begin{proof}
This is completely formal.
If $M$ and $N$ are two $\mathbb M$-modules, then the canonical map
\[
M \to \mathcal H\mathrm{om}_{\mathbb Z_\bullet} (N, M \otimes_{\mathbb Z_\bullet} N)
\]
(coming from adjunction) provides a map
\[
M^{\mathrm{cond}} \to \mathcal H\mathrm{om}_{\mathbb Z_\bullet} (N, (M \otimes_{\mathbb Z_\bullet} N))^{\mathrm{cond}} \simeq \mathcal H\mathrm{om}_{\underline {\mathbb Z}} (N^{\mathrm{cond}}, (M \otimes_{\mathbb Z_\bullet} N)^{\mathrm{cond}})
\]
or equivalenlty
\[
M^{\mathrm{cond}} \otimes_{\underline {\mathbb Z}} N^{\mathrm{cond}} \to (M \otimes_{\mathbb Z_\bullet} N)^{\mathrm{cond}}. \qedhere.
\]
\end{proof}

In order to understand the essential image of the functor $M \mapsto M^{\mathrm{cond}}$, we need another construction:

\begin{prop} \label{expfrm}
\begin{enumerate}
\item
If $\mathcal M$ is a light condensed abelian group, then the presheaf defined on compact metrizable spaces as
\[
\mathcal M_\bullet : S \mapsto \ker(\mathcal M(S \times \overline{\mathbb N})\overset \infty \to \mathcal M(S))
\]
is a light condensed abelian group.
\item The endofunctor $\mathcal M \mapsto \mathcal M_\bullet$ is fully faithful and preserves all limits and all colimits.
\end{enumerate}
\end{prop}

\begin{proof}
Everything is straightforwards besides preservation of colimits that follows from to part \eqref{intPr1} of theorem \ref{intPr} below according to the coming discussion.
\end{proof}

Even if we shall not use it, it is worth mentioning the following topos-theoretic description of the functor $\mathcal M \mapsto \mathcal M_\bullet$: if we denote by $\mathrm{Cond}$ the category of light condensed sets, then there exists a commutative diagram of topos
\[
\begin{tikzcd}
\mathrm{Cond} \ar[r, rightarrowtail , "j"] \ar[rd, equal] & \mathrm{Cond}_{/\underline{\overline{\mathbb N}}} \ar[d, "p"] & \mathrm{Cond}_{/\underline{\overline{\mathbb N}}}^* \ar[l, rightarrowtail , "i"'] \ar[ld, "q"']
\\ & \mathrm{Cond} &
\end{tikzcd}
\]
where the upper maps are open and closed complement induced by $\infty : * \rightarrowtail \overline{\mathbb N}$.
Then, we have
\[
\mathcal M_\bullet = q_*i^!p^{-1}\mathcal M.
\]
In particular, there exists an explicit adjoint given by $p_!i_*q^{-1}$ (but the coadjoint is more mysterious).

There exists yet another description.
Recall that we have a decomposition
\[
\mathbb Z \cdot \overline {\mathbb N} \simeq P \oplus \mathbb Z \cdot  \infty \quad \mathrm{with} \quad P := \mathbb Z \cdot \overline {\mathbb N}^*
\]
(Graev's free abelian group of countable rank) when $\overline {\mathbb N} : = \mathbb N \cup \{\infty\}$ is endowed with its profinite topology and pointed at infinity.
Then, we have
\[
\mathcal M_\bullet \simeq \mathcal H\mathrm{om}_{\underline {\mathbb Z}}(\underline P, \mathcal M).
\]
In particular, the functor $\mathcal M \mapsto \mathcal M_\bullet$ preserves all colimits if and only $\underline P$ is internally finitely presented projective (by definition).
This is exactly the content of the first assertion of theorem \ref{intPr} below.

\begin{prop} \label{comp}
\begin{enumerate}
\item
If $V$ is a topological abelian group, then $\underline V_\bullet = \underline {V_\bullet}$.
\item \label{PM}
If $M$ is an $\mathbb M$-module, then $(M^{\mathrm{cond}})_\bullet \simeq \underline M$ (for the canonical topology).
\end{enumerate}
\end{prop}

\begin{proof}
The first assertion follows from
\[
\mathcal H\mathrm{om}_{\underline {\mathbb Z}}(\underline {\mathbb Z \cdot \overline {\mathbb N}}, \underline V) \simeq \mathcal H\mathrm{om}(\underline{\overline {\mathbb N}}, \underline V) = \underline {\mathrm{Hom}_{\mathrm{cont}}(\overline {\mathbb N}, V)}
\]
and lemma \ref{Pbul}.
Since both functors preserve all colimits, the second assertion reduces to $\mathbb M^{\mathrm{cond}} \simeq \underline {\mathbb Z}^{\mathbb N}$.
\end{proof}

If $\mathcal M$ is a light condensed abelian group, then the shift $t$ on $\overline {\mathbb N}$ that sends $n$ to $n+1$ induces an endomorphism still denoted by $t$ of $\mathcal M_\bullet$ and $\mathcal M$ is said to be \emph{solid} if $1 - t$ is an automorphism.

\begin{prop}
If $M$ is an $\mathbb M$-module, then $M^{\mathrm{cond}}$ is solid.
\end{prop}

\begin{proof}
The functor $M \mapsto M^{\mathrm{cond}}$ is fortunately compatible with both actions of $t$.
Since $\mathbb M$ is complete for the $t$-adic topology, $1-t$ is invertible in $\mathbb M$ and $1-t$ is therefore bijective on any $\mathbb M$-module.
\end{proof} 

The condition for being solid is stable under all (limits and) colimits.
It follows that light solid abelian groups form a reflective (and coreflective) subcategory and one denotes by $\mathcal M \mapsto \mathcal M^\blacksquare$ the reflection (adjoint to inclusion) which is called \emph{solidification}.
Note that $\mathcal H \mathrm{om}_{\underline {\mathbb Z}}(\mathcal M, \mathcal N)$ is automatically solid when $\mathcal N$ is and there exists an adjoint given by $\mathcal M \otimes^\blacksquare \mathcal N := (\mathcal M \otimes_{\underline{\mathbb Z}} \mathcal N)^\blacksquare$ so that the category is closed symmetric monoidal.
Actually, the category of light solid abelian groups is equivalent to the category of $\mathbb M$-modules.
This may be shown as follows: the category of light solid abelian groups si a cocomplete abelian category with exact filtered colimits that admits, according to theorem part \eqref{intPr2} of theorem \ref{intPr} below, a finitely presented projective generator $\underline {\mathbb Z}^{\mathbb N}$.
Since $\mathbb M^{\mathrm{op}} \simeq \mathrm{End}_{\underline {\mathbb Z}}(\underline {\mathbb Z}^{\mathbb N})$, our claim is the easy case of Gabriel-Popescu theorem as explained in the appendix.
We can actually be more precise:

\begin{thm} \label{infMor}
The functor $M \mapsto M^{\mathrm{cond}}$ induces an equivalence between $\mathbb M$-modules and light solid abelian groups with inverse $\mathcal M \mapsto \mathcal M_\bullet(*)$.
\end{thm}

\begin{proof}
We need to check that, if $M$ is an $\mathbb M$-module, then $M^{\mathrm{cond}}_\bullet(*) \simeq M$ and that, if $\mathcal M$ is a solid abelian group then $\mathcal M \simeq \mathcal M_\bullet(*)^{\mathrm{cond}}$.
Since all our functor preserve colimits and $\underline{\mathbb Z^\mathbb N}_\bullet \simeq \underline{\mathbb Z^\mathbb N_\bullet} \simeq \underline {\mathbb M}$, both assertions reduce to
\[
\mathbb M^{\mathrm{cond}}_\bullet(*) \simeq \underline{\mathbb Z^\mathbb N}_\bullet(*) = \underline {\mathbb M}(*)\simeq \mathbb M
\]
and, thanks to part \eqref{intPr2} of theorem \ref{intPr} \emph{below}, to
\[
\underline{\mathbb Z^\mathbb N}_\bullet(*)^{\mathrm{cond}} \simeq \underline{\mathbb M}(*)^{\mathrm{cond}} \simeq \mathbb M^{\mathrm{cond}} \simeq \underline{\mathbb Z^\mathbb N}.\qedhere
\]
\end{proof}

As a consequence, we see that, if $\mathcal M$ is a light condensed abelian group, then $\mathcal M^{(\blacksquare)} = \mathcal M^\blacksquare_\bullet(*)$.

Recall that a \emph{light condensed ring} $\mathcal A$ is, equivalently, a sheaf of rings on the site of compact metrizable spaces and jointly surjective finite covering, a ring in the cartesian category of light condensed sets or a monoid in the monoidal category of light condensed abelian group.
Then, an \emph{$\mathcal A$-module} is, equivalently, an $\mathcal A$-module in the cartesian category of light condensed sets or an $\mathcal A$-object in the monoidal category of light condensed abelian group.
The next statement is then a formal consequence of theorem \ref{infMor}:

\begin{cor}
The functor $A \mapsto A^{\mathrm{cond}}$ (resp. $M \mapsto M^{\mathrm{cond}}$) induces an equivalence between $\mathbb M$-rings (resp. $A$-modules) and light solid rings (resp. $A^{\mathrm{cond}}$-modules ) with inverse $\mathcal A \mapsto \mathcal A_\bullet(*)$ (resp. $\mathcal M \mapsto \mathcal M_\bullet(*)$). \qed
\end{cor}

The results of this section rely on the following:

\begin{thm}[Clausen-Scholze] \label{intPr}
\begin{enumerate}
\item \label{intPr1}
The light condensed abelian group $\underline P$ is internally (and externally) finitely presented\footnote{aka \emph{compact}.} projective\footnote{When there are not enough projectives, a projective need not be internally projective.} .
\item \label{intPr2}
The light solid abelian group $\underline{\mathbb Z}^\mathbb N$ is a finitely presented projective generator (and $\underline P^\blacksquare \simeq \underline{\mathbb Z}^\mathbb N$).
\end{enumerate}
\end{thm}

\begin{proof}
\begin{enumerate}
\item
If $S$ is a compact metrizable space, then $\underline S$ is an internally finitely presentable light condensed set (qcqs object of a coherent topos).
It follows that $\underline{\mathbb Z} \cdot \underline S$ is an internally finitely presentable light condensed abelian group.
This is therefore also the case for $\underline P$ which is a direct factor of $\underline{\mathbb Z} \cdot \underline{\overline {\mathbb N}}$.
Next, one shows that, in the category of Stone metrizable spaces (aka light profinite sets), every non-empty object is injective.
This is used to prove internal projectivity of $P$ by hand: see for example \cite{Warn24} (theorem 5), \cite{ Camargo26} (theorem 2.3.3.(3)) or \cite{Kedlaya25} (proposition 5.4.8).
\item
It follows from the first part that $\underline P^\blacksquare$ is a finitely presented projective light solid abelian group.
Then, one builds a null sequence\footnote{Use $[n] - [n - 2^{v_2(n)}]$ where $[n] \in 2^\mathbb N$ denotes the binary expansion of $n$.} in $\mathbb Z \cdot 2^\mathbb N$ that provides an isomorphism $P^\blacksquare \simeq \underline {\mathbb Z} \cdot \underline 2^{\mathbb N\blacksquare}$ on solidifications.
Since $\underline {\mathbb Z} \cdot \underline {2^\mathbb N}$ is a generator of the category of light condensed abelian groups, it formally follows that $\underline P^\blacksquare$, is a generator of the category of light solid abelian groups.
To conclude, one shows that the coordinate map $\mathbb N \to \mathbb Z^\mathbb N$ extends to an isomorphism $\underline P^\blacksquare \simeq \underline {\mathbb Z}^\mathbb N$.
This is done in two steps\footnote{The defintion of the middle term is left to the imagination of the reader.}: $\underline P^\blacksquare \simeq \underline {\mathbb Z}^{\mathbb N, \mathrm{bd}, \blacksquare} \simeq\underline {\mathbb Z}^\mathbb N$ with two different kinds of arguments.
For the details, we send the reader to theorem 3.3.1 of \cite{ Camargo26} or to the course \cite{Kedlaya25}, or, of course, the online lectures \cite{ClausenScholze23}. \qedhere
\end{enumerate}
\end{proof}


\section{Appendix: Watts-Eilenberg/Gabriel-Popescu} \label{Appendix}

We do not really need the full theorems of Watts-Eilenberg and Gabriel-Popescu but mostly the context of their statements.

If $R$ is a ring, then a right \emph{$R$-module} in an additive category $\mathcal A$ is an object $G$ of $\mathcal A$ endowed with a homomorphism of rings $R^{\mathrm{op}} \to \mathrm{End}_{\mathcal A}(G)$.
There exists then a functor
\begin{equation} \label{Homfct}
\mathcal A \to \mathrm{Mod}_R, \quad M \mapsto \mathrm{Hom}_{\mathcal A}(G, M)
\end{equation}
that preserves all limits (that exist in $\mathcal A$).
If $\mathcal A$ is cocomplete (and the solution set condition is satisfied), then there exists an adjoint
\begin{equation} \label{Tensfct}
\mathrm{Mod}_R \to \mathcal A, \quad M \mapsto G \otimes_R M
\end{equation}
(where the tensor product on the right hand side must be seen as a bare notation).
The theorem of Watts-Eleinberg states that any functor that preserves all colimits has this form (theorem 3.1 of \cite{NymanSmith16}).

Now, an object $G$ in a category $\mathcal C$ is called a \emph{separator} (resp.\ a \emph{generator}, resp.\ a \emph{projective} object, resp.\ a \emph{finitely presented}\footnote{aka compact.} object) if the functor
\begin{equation} \label{fHom}
X \mapsto \mathrm{Hom}(G, X)
\end{equation}
is faithful (resp.\ is conservative, resp.\ preserves epimorphisms, resp.\ preserves filtered colimits).
When $\mathcal C$ is cocomplete with finite limits, a generator $G$ is always a separator and any $X \in \mathcal C$ is then a colimit of copies of $G$.
All three conditions are actually equivalent if, moreover, $\mathcal C$ is \emph{balanced} (a morphism is an isomorphism if and only if it is a monomorphism and an epimorphism).
In the case $\mathcal C$ is an abelian category, $G$ is projective (resp. finitely presented projective) if and only if the functor \eqref{fHom} preserves all finite colimits (resp. all colimits).

The theorem of Gabriel-Popescu (corollary 4.10 of \cite{Popescu73}) states that, if $\mathcal A$ is a cocomplete abelian category with exact filtered colimits, $G$ is a generator of $\mathcal A$ and $R := \mathrm{End}(G)^{\mathrm{op}}$, then the functor \eqref{Homfct} is fully faithful and the functor \eqref{Tensfct} is exact.
Moreover, they induce an equivalence if and only if $G$ is finitely presented projective.

As a baby example, you can consider the category of (left) $\mathbb M_r$-modules where $\mathbb M_r$ denotes the ring of $r \times r$-matrices.
Since
\[
\mathbb M_r \simeq \underbrace{\mathbb Z^r \oplus \cdots \oplus \mathbb Z^r}_{r\ \mathrm{times}}
\]
as $\mathbb M_r$-modules, it is easily checked that $\mathbb Z^r$ is a finitely presented projective generator.
Since $\mathrm{End}_{\mathbb M_r}(\mathbb Z^r)$ identifies with the center of $\mathbb M_r \simeq \mathrm{End}_{\mathbb Z}(\mathbb Z^r)$ which is nothing but the subset of diagonal matrices, and therefore isomorphic to $\mathbb Z$, we recover \emph{Morita equivalence} between (left) $\mathbb M_r$-modules and abelian groups.

\newpage

\addcontentsline{toc}{section}{References}
\printbibliography

\Addresses

\end{document}